\documentclass[a4paper,10pt]{amsart}

\usepackage[utf8]{inputenc}
\usepackage{amsmath, amsthm, amssymb, comment}
\usepackage[dvips]{graphicx}
\usepackage{overpic}
\usepackage{enumitem}
\usepackage[colorlinks=true, citecolor=red]{hyperref}
\usepackage{pinlabel}

\usepackage{caption}
\usepackage{subcaption}

\usepackage{thm-restate}

\usepackage[margin=3.5cm]{geometry}

\usepackage{color}

\newtheorem{theorem}{Theorem}[section]
\newtheorem{corollary}[theorem]{Corollary}
\newtheorem{proposition}[theorem]{Proposition}
\newtheorem{definition}[theorem]{Definition}
\newtheorem{lemma}[theorem]{Lemma}
\newtheorem{claim}[theorem]{Claim}

\newtheorem*{theorem*}{Theorem}
\newtheorem*{proposition*}{Proposition}
\newtheorem*{definition*}{Definition}
\newtheorem*{lemma*}{Lemma}
\newtheorem*{claim*}{Claim}
\newtheorem*{corollary*}{Corollary}
\newtheorem*{convention*}{Convention}
\newtheorem{observation}[theorem]{Observation}

\theoremstyle{definition}

\theoremstyle{remark}
\newtheorem{rem}[theorem]{Remark}
\newtheorem{remark}[theorem]{Remark}

\newtheorem*{rem*}{Remark}

\newcommand{\wt}[1]{\widetilde{#1}}

\newcommand\bR{\mathbb R}
\newcommand\R{\mathbb R}

\newcommand\bZ{\mathbb Z}

\DeclareMathOperator{\id}{id}

\newcommand{\flow}{ \varphi }
\newcommand{\hflow}{ \wt{\varphi} }

\newcommand\orb{ \mathcal O_{\flow} }
\newcommand\fs{\mathcal F^{s} }
\newcommand\hfs{\widetilde{\mathcal F}^{s} }
\newcommand\fu{\mathcal F^{u} }
\newcommand\hfu{\widetilde{\mathcal F}^{u} }
\newcommand\bfs{\bar{\mathcal F}^{s} }
\newcommand\bfu{\bar{\mathcal F}^{u} }

\newcommand\fss{\mathcal F^{ss} }

\newcommand{\cF}{\mathcal{F}}

\newcommand{\cO}{\mathcal{O}}

\newcommand{\acts}{\curvearrowright}

\newcommand{\Homeo}{\mathrm{Homeo}}

\newcommand{\Stab}{\mathrm{Stab}}
\newcommand{\Aut}{\mathrm{Aut}}

\newcounter{notes}%[page]   %Le 2eme argument fait reinitialiser les numeros de notes a chaque page
\title{Reconstructing flows from the orbit space}

\author[Thomas Barthelm\'e]{Thomas Barthelm\'e}
\address{Queen's University, Kingston, Ontario}
\email{thomas.barthelme@queensu.ca}
\urladdr{sites.google.com/site/thomasbarthelme}

\author[Sergio Fenley]{Sergio Fenley}
 \address{Florida State University, Tallahassee, FL}
 \email{sfenley@fsu.edu}
\author[Kathryn Mann]{Kathryn Mann}
 \address{Cornell University, Ithaca, NY}
 \email{k.mann@cornell.edu}
\urladdr{https://e.math.cornell.edu/people/mann}

\begin{document}

\begin{abstract}
We give some simple conditions under which a group acting on a bifoliated plane comes from the induced action of a pseudo-Anosov flow on its orbit space. An application of the strategy is a less technical proof of a result of Barbot that the induced action of an Anosov flow on its orbit space uniquely determines the flow up to orbit equivalence. In another application, we recover an expansive flow on a 3-manifold from the action of a group on a \emph{loom space} as defined by Schleimer and Segerman.
\end{abstract}
\maketitle

\section{Introduction} 
Following work of Barbot \cite{Bar95a}, the second
author \cite{FEn94} and Mosher \cite{FM01}, a pseudo-Anosov flow on a compact 3-manifold $M$ gives rise to an action of $\pi_1(M)$ on a topological plane, called the {\em orbit space} with two topologically transverse, invariant foliations (possibly with prong singularities) induced by the stable and unstable foliations of the flow.  
In this work, we give a simple condition under which an action of an arbitrary torsion-free group $G$ on a bifoliated plane is the one induced by a pseudo-Anosov flow on a $3$-manifold $M$, thus implying that $G$ is a 3-manifold group.  This condition is necessary and sufficient for the class of transversally orientable pseudo-Anosov flows. For noncompact manifolds, we treat the case of foliation-preserving expansive flows. (The notions of expansive and pseudo-Anosov flows are known to coincide in the compact case. In the non-compact case, there is not a well-developed theory of such flows, see Remark \ref{rem_expansive_vs_pA_noncompact} for more details).
As described below, our condition generalizes Thurston's notion of an {\em extended convergence group} from \cite{Thu97}, which corresponds to the special case when the plane has a particular global structure called {\em skew}.  

We also use the same strategy to give a simple proof of a theorem of Barbot that (for Anosov flows) the action of $\pi_1(M)$ on the orbit space determines the flow up to orbit equivalence.  This does not use any hypothesis that $G$ preserves orientation of $P$ or leafwise orientation of either foliation. 

As a second application, we show in Section \ref{sec_applications} that groups of automorphisms of {\em loom spaces} (preserving orientation) are 3-manifold groups and their action naturally gives rise to an expansive flow on a 3-manifold which is atoroidal in the sense that any $\bZ \times \bZ$ subgroup of $G$ fixes a (unique) cusp.   
After an earlier version of this work was circulated, we learned that such groups of automorphisms of loom spaces were already known to be 3-manifold groups by the work of Baik--Jung--Kim \cite[Theorem 17.15]{BJK25}. Our approach gives an alternative argument which yields a manifold together with an associated expansive flow, and is meant as an illustration of the use of the main result.

Throughout this work, unless otherwise stated, pseudo-Anosov flows are assumed to be {\em topologically pseudo-Anosov} and not necessarily have additional smooth structure or strong stable/unstable foliations.  See \cite[Definition 1.1.10]{BM_book} for a precise definition and discussion of related definitions. 

\subsection*{Statement of results} 
Let $(P, \cF_1, \cF_2)$ denote a topological plane with two topologically transverse 1-dimensional foliations, possibly with prong singularities but at most one on any given leaf.  Such a structure is called a {\em bifoliated plane}, and $\cF_i(x)$ denotes the leaf of $\cF_i$ containing $x \in P$. 
If a foliation $\cF_i$ is equipped with a leafwise orientation (varying continuously between leaves), we denote by $\cF_i^>(x)$ the connected component of $\cF_i(x)\smallsetminus \{x\}$ on the positive side of $x$ (or connected components, if $x$ is a singularity). Notice that such an orientation exists if and only if $\cF_i$ has no odd prong singularities. In particular, $\cF_1$ admits a leafwise orientation if and only if $\cF_2$ admits one.

\begin{definition} \label{def_W1}
Assuming $(P, \cF_1, \cF_2)$ is nonsingular,  define the space 
\[W_1^> := \{(x,t) \in P \times P \mid t \in \cF_1^>(x)\} \]
equipped with the subset topology from $P \times P$.
\end{definition} 
Topologically, it is easy to show that $W_1^> \cong \bR^3$. The space $W^>_2$ is defined analogously.  
Since the foliations may always be relabeled, we make the convention of stating all results in terms of $W_1^>$, but obviously the roles of $1$ and $2$ may be swapped in any result.  
When $\cF_i$ are leafwise oriented but $(P, \cF_1, \cF_2)$ has singularities, we will modify the definition of $W^>_i$ to account for the presence of prongs by passing to a quotient -- see the discussion after Theorem \ref{thm_compact_singular}.  

By \cite{Fen98}, if a pseudo-Anosov flow on a compact manifold is not orbit-equivalent to the suspension of a hyperbolic linear map of the torus, then its orbit space has no {\em infinite product regions}, precisely: 
 \begin{definition} 
 If $(P, \cF_1, \cF_2)$ is a bifoliated plane, a {\em $\cF_i$-infinite product region} is a subset of $P$ which is the image of a proper embedding from $[0,1] \times [0, \infty) \subset \bR^2$ with its product foliation into $P$, such that images of rays $\{p\} \times [0, \infty)$ are in $\cF_i$. 
 \end{definition}
Suspension Anosov flows are well understood and their orbit spaces are trivially bifoliated planes.  Thus, it is natural to exclude such examples, and we do so by working on planes without infinite product regions in at least one foliation. 
With this set-up, 
we give a necessary and sufficient condition for groups of automorphisms (foliation-preserving homeomorphisms) to be induced from transversally orientable pseudo-Anosov flows on compact 3-manifolds.  
Here and in what follows, $\Aut(P)$ denotes the group of homeomorphism of $P$ that preserve each foliation (sending leaves to leaves),  $\Aut^+_i(P)$  is the subgroup of homeomorphisms preserving a leafwise orientation of $\cF_i$, in contexts where such an orientation exists, and $\Aut^+(P)$ those preserving orientations of both foliations (if such exist).  

Our first result is a reconstruction theorem, for simplicity we state it first for nonsingular planes: 
\begin{theorem} \label{thm_compact}
Let $(P, \cF_1, \cF_2)$ be a nonsingular bifoliated plane with no $\cF_1$-infinite product region, and $G$ a torsion-free subgroup of $\Aut^+_1(P)$. 
If $G$ acts properly discontinuously and cocompactly on $W^>_1$, then $M = W^>_1/G$ is a compact $3$-manifold equipped with a topological Anosov flow $\flow$, such that $(P, \cF_1, \cF_2)$ is the orbit space of $\flow$, and the action of $G$ agrees with the  action of $\pi_1(M)$ induced by $\flow$.

Consequently, $P$ also has no $\cF_2$-infinite product regions, and
if $G$ also preserves orientation on $\cF_2$, then 
$G$ acts properly discontinuously and cocompactly on $W^>_2$ as well.  
\end{theorem}
In the statement above, by ``$(P,\cF_1,\cF_2)$ is the orbit space of $\varphi$'' we mean
that there is a natural (and obvious from the construction) homeomorphism
from the orbit space of $\flow$ to $P$, which sends the stable and unstable
foliations in the orbit space to the pair $( \cF_1, \cF_2 )$.
The same theorem holds by symmetry with the roles of $W^>_1$ and $W^>_2$ reversed. 

For singular bifoliated planes, we obtain the following:

\begin{theorem} \label{thm_compact_singular}
Let $(P, \cF_1, \cF_2)$ be a possibly singular bifoliated plane with leafwise orientations and no $\cF_1$-infinite product region, and $G$ a torsion-free subgroup of $\Aut^+_1(P)$. 
Assume that for any prong singularity $p\in P$, the stabilizer of $p$ in $G$ is cyclic.
Then there exists a topological space $W_1^\ast$, homeomorphic to $\bR^3$, obtained as a quotient of $W_1^>$
with a natural induced action of $G$.

If $G$ acts properly discontinuously and cocompactly on $W^\ast_1$, then $M = W^\ast_1/G$ is a compact $3$-manifold equipped with a pseudo-Anosov flow $\flow$, such that $(P, \cF_1, \cF_2)$ is the orbit space of $\flow$, and the action of $G$ agrees with the  action of $\pi_1(M)$ induced by $\flow$.

Consequently, $P$ also has no $\cF_2$-infinite product regions, and
if $G$ also preserves orientation on $\cF_2$, then 
$G$ acts properly discontinuously and cocompactly on $W^\ast_2$ as well.    
\end{theorem}

The idea to obtain the space $W_1^\ast$ in the singular case is simple: The reason we cannot just consider $W_1^>$ is that this space is not homeomorphic to $\bR^3$ as, for any $2n$-prong $p$, $\cF_1^>(p)$ consists of $n$ disjoint rays. Thus, in order to obtain a $3$-dimensional space, one needs to identify the different rays of $\cF_1^>(p)$. This identification can be done arbitrarily, as long as it is invariant under the group action. See Definition \ref{def_W_singular} for the precise statement.  
The hypothesis that $\mathrm{Stab}_G(p)$ is cyclic is used to ensure that there exists such an identification. (We will also give other conditions on $G$ that ensures that $\mathrm{Stab}_G(p)$ is cyclic, see Lemma \ref{lem_cyclic_stabilizer}.)   

When $P$ is nonsingular, the definition of $W_i^\ast$ is as a trivial quotient of $W_i^>$ and thus they are equal.  Going forward, to streamline theorem statements we use this convention so that $W_i^\ast =W_i^>$ in the nonsingular case.  The two separate notations are re-introduced in proofs where it is important to keep track of prongs or not.  

Generalizing the above results to the noncompact case, we have: 
\begin{theorem} \label{thm_noncompact}
Under the hypotheses of Theorem \ref{thm_compact} or Theorem \ref{thm_compact_singular}, if the action of $G$ on $W_1^\ast$ is properly discontinuous but not cocompact, then $M = W_1^\ast/G$ admits a foliation-preserving expansive flow whose orbit space is $(P, \cF_1, \cF_2)$.
\end{theorem}

\begin{rem}
We note that, instead of assuming that $G$ is torsion free, one can equivalently assume that the action of $G$ on $W_1^\ast$ is free. 
This hypothesis is used only to ensure that $W^\ast_1/G$ is a manifold, not an orbifold.  
Furthermore, in the case of a nonsingular bifoliated plane, if $G$ preserves both transverse orientations, then it is not necessary to suppose that $G$ is torsion free to prove that the action of $G$ on $W^\ast_1$ is free: this is an easy consequence of proper discontinuity and preservation of orientation.  See Lemma \ref{lem_torsion_free}. 
\end{rem}

As a converse to Theorems \ref{thm_compact} and \ref{thm_compact_singular} we show:
\begin{theorem}  \label{thm_converse} 
Suppose that $\flow$ is a transversally orientable pseudo-Anosov flow on a compact $3$-manifold $M$,  not orbit-equivalent to a suspension of an Anosov diffeomorphism. Then the orbit space is a bifoliated plane without infinite product regions, and $\pi_1(M)$ acts properly discontinuously and cocompactly on the spaces $W^\ast_1$ and $W^\ast_2$ associated to the stable and unstable foliations of its orbit space.  
\end{theorem}

\subsection{Conditions for proper discontinuity} 
In practice, it is not always easy to check that a given action of $G$ on a bifoliated plane induces a properly discontinuous action on the space $W_1^\ast$ (or $W^2_\ast$), and it can be useful to have a condition that can be read off of the (local) dynamics on $P$.  
This is much like the notion of {\em convergence group} of Gehring and Martin \cite{GM87} which has a definition in terms of a properly discontinuous action on a space of triples, and an equivalent definition in terms of ``convergence sequences", both of which are useful for different purposes.  

In this spirit, our next results give examples of conditions to obtain proper discontinuity.   We introduce two properties abstracted from the dynamics associated with pseudo-Anosov flow.  The first is an orbit space version of the Anosov closing lemma. 

\begin{definition}
Let $(P, \cF_1, \cF_2)$ be a nonsingular bifoliated plane.  A group $G < \Aut(P)$ {\em has the closing property} if for all $x\in P$, and each neighborhood $U_x$ of $x$, there exists a smaller neighborhood $V_x\subset U_x$, such that, if $g(V_x) \cap V_x \neq \emptyset$, then $g$ has a fixed point in $U_x$.  
\end{definition} 
See Definition \ref{def_closing_prop_general_case} for the statement of the closing property in the presence of singular points.
The classical closing lemma for (pseudo)-Anosov flows implies that all orbit space actions satisfy this property (see \cite[Proposition 1.4.7]{BM_book}). 

The next condition is a version of ``uniform hyperbolicity'' at fixed points. 
\begin{definition}\label{def_uniform_hyperbolicity}
A group $G < \Aut(P)$  has \emph{hyperbolic fixed points} if for every $x\in P$ fixed by some nontrivial $g\in G$, either $g$ or $g^{-1}$ acts as a topological contraction on $\cF_1(x)$ and a topological expansion on $\cF_2(x)$.

We say $G$ has \emph{uniformly hyperbolic fixed points} if additionally, for any compact rectangle\footnote{By compact rectangle, we mean the image of an embedding of $[0,1]\times [0,1]$ in $P$ sending the horizontal, resp.~vertical, leaves of the unit square to the $\cF_1$, resp.~$\cF_2$, leaves.} $R \subset P$ and sequence $g_n$ with fixed points in $R$, if $\cF_i(g_n \mathring{R}) \supset \cF_i(g_{n+1}\overline{R})$ for all $n$ (with $i = 1$ or $2$), then $\bigcap_n \cF_i(g_n R)$ is a single leaf. 
\end{definition} 
For Anosov flows on compact $3$-manifolds, this uniform property corresponds to the fact that any sequence of periodic orbits which intersect a given compact transverse disk $\tau$ must contain longer and longer periodic orbits (possibly including higher and higher powers of the same periodic orbits), thus fixed points of (powers of) the first return map that will have stronger and stronger hyperbolicity. See the proof of Theorem \ref{thm_converse} in Section \ref{sec_closing} for a detailed discussion of this.

We show: 
\begin{theorem} \label{thm_closing_implies_prop_disc}
Let $(P, \cF_1, \cF_2)$ be a bifoliated plane and $G < \Aut^+_1(P)$.
If $G$ has the closing property and uniformly hyperbolic fixed points, then $G$ acts freely and properly discontinuously on $W^\ast_1$. 
\end{theorem}
The result above holds also in the singular case, with the appropriate definition of closing property (Definition \ref{def_closing_prop_general_case}), since in this context the stabilizer of any point is trivial or cyclic (Lemma \ref{lem_cyclic_stabilizer}), and thus the appropriate space $W_1^\ast$ can be constructed.

Combining this with the previous theorem gives: 
\begin{corollary} 
Let $(P, \cF_1, \cF_2)$ be a bifoliated plane with no $\cF_1$-infinite product region.  Then any $G < \Aut^+_1(P)$  with the closing property and uniformly hyperbolic fixed points  is a
 3-manifold group, and $W^\ast_1/G$ admits an expansive flow whose orbit space is $(P, \cF_1, \cF_2)$. If $W^\ast_1/G$ is compact then the flow is pseudo-Anosov. 
 \end{corollary} 

Transitive pseudo-Anosov flows on a compact 3-manifold $M$ are characterized by the property that the set of points in the orbit space fixed by nontrivial elements of $\pi_1(M)$ is dense.  Thus, groups of automorphisms of bifoliated planes with a dense set of fixed points have become an important class to study.   In this case, we give below an even simpler condition to ensure uniformly hyperbolic fixed points: it suffices to assume that the fixed points are hyperbolic and that their orbit under $G$ is closed and discrete. In terms of flows, this condition is just saying that fixed points in the orbit space corresponds to hyperbolic periodic (hence compact) orbits.
\begin{theorem}\label{thm_closing_plus_transitive_implies_prop_disc}
Let $G < \Aut^+_1(P)$ be a group such that $\{ x \in P : \exists g \neq \mathrm{id} \text{ with } g(x) = x\}$ is dense in $P$.  Assume
\begin{enumerate}[label=(\roman*)]
\item $G$ has the closing property,
\item\label{item_thm_hyperbolic} $G$ has hyperbolic fixed points, and
\item For any $x\in P$ fixed by some nontrivial $g\in G$, $G\cdot x$ is closed and discrete in $P$.
\end{enumerate} 
Then $G$ acts properly discontinuously and freely  on $W^\ast_1$.
\end{theorem}

Notice that in \ref{item_thm_hyperbolic} of the above theorem we do not assume 
uniformity of hyperbolic fixed points. Once again, this result also holds without additional assumptions in the singular case.

\subsection*{Extended convergence groups}
Theorems \ref{thm_compact} and \ref{thm_converse} can be seen as a generalization to any transversally orientable pseudo-Anosov flow of Thurston's characterization of transversally orientable skew Anosov flows in terms of extended convergence groups \cite{Thu97}.   
A {\em skew} Anosov flow is one for which (say) the stable foliation
lifts to the universal cover to a foliation which has leaf space homeomorphic
to the reals, and the flow is not orbitally equivalent to a suspension
Anosov flow.

The \emph{skew plane} is the open region between $y=x$ and $y=x+1$ in $\bR^2$, with $\cF_1$ and $\cF_2$ being the horizontal and vertical foliations
respectively.  
For Thurston \cite{Thu97}, an {\em extended convergence group} is a subgroup of $\Homeo^+(\R)$, commuting with translation by $1$,  that acts properly discontinuously on the space $T:= \{(u,v,w) : u < v < w < u+1\}$.\footnote{There is a typo in the definition of extended convergence group in \cite{Thu97}, the action must be properly discontinuous on our space $T$, which corresponds to the space $\widetilde T$ in the notations of \cite{Thu97}, and not on $\widetilde{\bar T}$ as written in \cite[Definition 7.2]{Thu97}.} Considering $\R$ as the ``lower boundary" $y=x$ of the standard model for the diagonal skew strip, $\R$ is simultaneously identified with the leaf space of $\cF_1$ and of $\cF_2$.   There is an obvious homeomorphism from $T$ to $W^>_1$ (up to choice of orientation) as follows.  For $(u,v,w) \in T$, consider the leaf $l^1_w$ of $\cF_1$ corresponding to $w$; the leaves of $\cF_2$ corresponding to $u$ and $v$ intersect $l^1_u$ at unique points (say, $x$ and $y$ respectively), and we send $(u, v,w)$ to $(x, y) \in W^>_1$.   

Thurston's observation was that cocompact extended convergence groups are precisely those which come from orbit space actions of transversally orientable skew Anosov flows on compact 3-manifolds. However, his proof in \cite{Thu97} is very different from the one we give here.  

\subsection*{Applications} 
The definition of $W^>_1$ and its generalization to the prong case was inspired by the discovery of a simple ``constructive" proof of a theorem of Barbot that actions on orbit spaces determine an Anosov flow up to orbit equivalence, stated below (Theorem \ref{thm_action_determines_OEflow}).  We begin this paper by presenting this proof, to serve as motivation and to give a simple illustration of our arguments. 

In Section \ref{sec_applications}, we treat another application, describing the structure of automorphism groups of \emph{loom spaces} (see Definition \ref{def_loom_space}) which were introduced in \cite{SS24} as a class of bifoliated planes including those induced by a veering triangulation of a $3$-manifold. We show how one can verify the hypothesis of our theorems in this case, with surprisingly little assumption on the group action:

\begin{restatable}[Loom spaces give expansive flows]{theorem}{loomthm}\label{thm_loom_space}
Let $G< \Aut_1^+(P)$ with $(P,\cF_1,\cF_2)$ a loom space.   Then $G$ acts properly discontinuously and freely on $W_1^>$, so $G \cong \pi_1(M)$ for some 3-manifold $M$ and $W^>_1/G \cong M$ admits an expansive flow whose orbit space is $(P, \cF_1, \cF_2)$. Moreover, $M$ is ``atoroidal'', in the sense that any $\bZ^2$ subgroup of $\pi_1(M)$ fixes a unique cusp. When $G$ is finitely generated we prove the stronger fact that any $\pi_1$-injective torus or Klein bottle in $M$ is boundary parallel.
\end{restatable}

\begin{remark}
As outlined earlier, the fact that $G$ is a 3-manifold group was shown previously in \cite{BJK25}, the new content of this result is the existence of the expansive flow.
\end{remark}

\begin{remark}
There is another case of bifoliated planes studied in the literature which could readily fit our framework: Iakovoglou in \cite{Iak22} introduced the notion of bifoliated planes admitting a {\em markovian family}. While the focus in \cite{Iak22} is to start from the orbit space of an Anosov flow, one can instead start with an abstract markovian family on a bifoliated plane, and show that it will have to come from an Anosov (or expansive in the non-compact case) flow by adapting the strategy we use for loom spaces. Since this abstract version has not yet been written, nor has yet attracted as much attention as the use of veering triangulations, we did not develop that here. 
\end{remark}

\subsection*{Outline} 
Section \ref{sec_Barbot} sets the stage for the paper, giving a constructive proof of Barbot's theorem. 
The proof of Theorems \ref{thm_compact} and \ref{thm_noncompact} for the nonsingular case are done in Section \ref{sec_nonsingular_main_theorem}. The nonsingular version of Theorem \ref{thm_closing_implies_prop_disc} is treated in Section \ref{sec_proper_discontinuity}. In Section \ref{sec_pA_case} we describe the modifications needed to treat the singular case and deduce Theorem \ref{thm_converse}.  Finally, Section \ref{sec_applications} contains the application to loom spaces.

\begin{remark}[Concurrent work by Baik, Wu and Zhao]
A draft version of this note was circulated in 2024, and the definition of $W^>_1$ applied to the proof of Barbot's theorem as well as a sketch of its generalization were presented at the conference {\em Beyond Uniform Hyperbolicity} in June 2023.  Around the same time, Baik, Wu and Zhao independently were working on very closely related results, now available in the paper \cite{BWZ24}. More precisely, rephrased into the terminology that we use here, Theorem 1.1 of \cite{BWZ24} gives similarly to our Theorem \ref{thm_noncompact} that if a group $G<\Aut_1^+(P)$ acts freely and properly discontinuously on $W_1^>$, then one gets a flow preserving two transverse foliations in $W_1^>/G$. While we assume no infinite product regions and deduce expansivity of the flow, they instead assume that elements of $g$ fixing a leaf must act with a unique hyperbolic fixed point on it, and deduce that the behavior on cylindrical leaves of the flow in $W_1^>/G$ is like that of a topological Anosov flow. This weaker notion of (pseudo)-Anosov flow is what they introduce as a \emph{reduced pseudo-Anosov flow} (Definition 2.17 of \cite{BWZ24}).
We encourage the reader to consult their work for this alternative perspective.
\end{remark}

\subsection*{Acknowledgements} 
TB thanks Théo Marty for a discussion in 2023 in which he suggested the space $W^>_1$ has a good potential model for an Anosov flow.
TB was partially supported by the NSERC (ALLRP 598447 - 24 and RGPIN-2024-04412). SF was partially supported by NSF DMS-2054909. KM was partially supported by NSF CAREER grant DMS-1933598 and a Simons foundation fellowship.
The authors thank H.~Baik for pointing out the paper \cite{BJK25}. We also thank him as well as S.~Schleimer and H.~Segerman for remarking that an assumption in Theorem \ref{thm_loom_space} in an earlier version of this paper was in fact unnecessary. Finally, we thank S.~Taylor for his comments and for pointing us to the reference \cite{Tsa23}.

\section{Motivation: a (re)-constructive proof of Barbot's theorem}  \label{sec_Barbot}
In this section $\flow$ denotes a topological Anosov flow on a compact 3-manifold $M$.  Its {\em orbit space}, denoted $\orb$ is the quotient space of $\wt M$ by the equivalence relation collapsing each orbit of the lifted flow $\hflow$ to a point.  By \cite{Bar95a,FEn94} this is a topological plane (the generalization to the pseudo-Anosov case is due to \cite{FM01}), and the action of $\pi_1(M)$ on $\wt M$ descends to an action on this plane by homeomorphisms.  
The 2-dimensional weak-stable and weak-unstable foliations $\fs$ and $\fu$ for $\flow$ lift to foliations $\hfs$ and $\hfu$ on $\wt M$, which descend to 1-dimensional foliations $\bfs$ and $\bfu$ on $\orb$ (also called stable/unstable
respectively) preserved by the action of $\pi_1(M)$.

Barbot showed that the action of $\pi_1(M)$ on $\orb$ determines the flow up to orbit equivalence, as follows: 

\begin{theorem}[Barbot \cite{Bar95a}, Théorème 3.4]\label{thm_action_determines_OEflow} 
Let $\flow$ and $\psi$ be  Anosov flows on $M$.  Suppose there exists an isomorphism $f_\ast \colon \pi_1(M) \rightarrow \pi_1(M)$ and a homeomorphism  $\bar f \colon \orb \rightarrow \cO_{\psi} $ sending the stable foliation of $\flow$ to that of $\psi$, which is $f_\ast$-equivariant, i.e., 
 \[ \bar f(g \cdot x) = f_\ast(g) \cdot \bar f(x)\] 
 for all $x \in \mathcal{O}_\flow$ and $g\in \pi_1(M)$.  Then $\flow$ and $\psi$ are orbit equivalent by a homeomorphism $f$ preserving direction of the flow and inducing $f_\ast$ on $\pi_1(M)$ and $\bar f$ on $\orb$.  
\end{theorem} 
Barbot's statement assumes the flows are smooth Anosov, but the proof works in the topological case.  

\begin{remark} 
The assumption that $\bar f$ respects stable foliations is used only to ensure that $f$ preserves direction, i.e., orientation of flow lines.  Using the connectedness of the plane, one can easily show from the dynamics of the action of $\pi_1(M)$ on orbit spaces that any equivariant homeomorphism must send the pair of stable/unstable foliations for one flow to the pair for the other, but might swap stable and unstable (see \cite[Proposition 1.3.19]{BM_book}).  Sending stable to stable ensures that the direction of the flow is not reversed.  
\end{remark} 

Barbot's original proof uses Haefliger's classifying spaces 
for the holonomy groupoid of the orbit foliations (see \cite{Haefliger}).  We give a proof that shows the flow can be canonically reconstructed from the action of $\pi_1(M)$ on the orbit space.  
The case where one of the foliations is {\em transversally orientable} is particularly simple, so we state this first.  The general case is a small modification, done in Theorem \ref{thm_nonorientable_model_flow}.

If $\cF^u$ (for example) is transversally oriented, this induces a $\pi_1(M)$-invariant transverse orientation of the one-dimensional foliation $\bfu$ on $\orb$, and hence a $\pi_1(M)$-invariant leafwise orientation on $\bfs$.  
Fixing such a choice, for $x\in \orb$, we let $\bfs_>(x)$ be the connected component of $\bfs(x)\smallsetminus \{x\}$ on the positive side of $x$, and define 
\[
W^>_s :=\{ (x,y) \in \orb \times \orb \mid y \in \bfs_>(x)\},
\]
equipped with the subset topology from $\orb \times \orb$.  
This has a natural action of $\pi_1(M)$, induced from the diagonal action $g(x,y) = (g(x), g(y))$ on $\orb \times \orb$, as well as a foliation whose leaves are the subsets of $W^>_s$ with constant first-coordinate.  

\begin{theorem}[Reconstruction theorem, special case] \label{thm_transverse_orientable_model_flow}
Let $\pi_1(M)\curvearrowright (\orb,\bfs,\bfu)$ be the induced action of an Anosov flow with an invariant leafwise orientation of $\bfs$. Let $\Psi$ be the constant first-coordinate foliation on $W^>_s$.  

Then the action of $\pi_1(M)$ on $W^>_s$ is properly discontinuous and free, 
$W^>_s/\pi_1(M)$ is homeomorphic 
to $M$, and $\Psi$ descends to a $1$-dimensional foliation on $W^>_s/\pi_1(M)$ such that any flow-parametrization of $\Psi$ is orbit equivalent to $\flow$. 
\end{theorem}
The same holds replacing $\fs$ with $\fu$, and defining the analogous space $W^>_u$.

The proof will use strong stable foliations for $\flow$ and an adapted metric on it.  The existence of this is classical for smooth Anosov flows see, e.g., \cite[Prop 5.1.5]{FH19}. For topological Anosov flows, one can always find an orbit equivalent one that will admit a strong stable or unstable foliation, but maybe not both, using the following result: 
\begin{proposition}[\cite{BFP23}, Corollary 5.23 and \cite{Potrie25}, Proposition 5.3]\label{prop_strong_stable}
If $\flow$ is a topological Anosov flow on $M$, there exists an orbit equivalent flow $\psi$  such that $\psi$ admits a strong stable invariant distribution $E^{ss}$ and an adapted metric on $M$, i.e., such that for all $v\in E^{ss}$ and $t>0$, we have 
$ \lVert D\flow^t(v)\rVert \leq e^{-t}\lVert v \rVert$.
\end{proposition}
The proof of this proposition has an error in \cite{BFP23}, but a correction is given in \cite[Proposition 5.3]{Potrie25}.  

\begin{proof}[Proof of Theorem \ref{thm_transverse_orientable_model_flow}]
Let $\flow$ be an Anosov flow on $M$ satisfying the hypotheses of Theorem \ref{thm_transverse_orientable_model_flow}. By Proposition \ref{prop_strong_stable}, we may assume that $\flow$ admits a strong stable distribution and  we have an associated adapted metric.

We define a map $h\colon W^>_s \to \wt{M}$ as follows:  A point $(x,y) \in W^>_s$ specifies two orbits $x$ and $y$, on the same weak-stable leaf.  Let 
$h(x,y)$ be the point on the orbit $x$ such that the distance along $\fss$ to the orbit $y$ is exactly 1 unit.  
This specifies a unique point thanks to our choice of adapted metric, where strong stable leaves are uniformly contracted under the flow.  
It is easy to see that $h$ is bijective and continuous, with continuous inverse, and thus a homeomorphism.  

We now show that $h$ is $\pi_1(M)$--equivariant, where the action of $\pi_1(M)$ on $W^>_s$ is the (diagonal) action induced from the orbit space.  Given $g \in \pi_1(M)$ and $(x,y) \in W^>_s$, by definition $h(g(x), g(y))$ is the point on the orbit of $g(x)$ whose distance along the (lifted) strong stable foliation is 1 from the orbit $g(y)$.  Since deck transformations act by isometries on $\wt{M}$ and preserve foliations, this is simply the image under $g$ of the point on orbit $x$ distance $1$ from the orbit $y$, in other words equal to $g( h(x,y))$.  

Thus, $h$ descends to a homeomorphism $\bar{h}\colon W^>_s/\pi_1(M) \to M$.  The constant first-coordinate foliation on $W^>_s$ is invariant under $h$ and its image under $\bar{h}$ is exactly the foliation by orbits of $\flow$.  This completes the proof.  
\end{proof}

Note that this already proves Theorem \ref{thm_converse} for Anosov flows.

Barbot's Theorem \ref{thm_action_determines_OEflow} now follows immediately (in the transversally orientable case): if $\flow$ and $\psi$ are two transversally orientable flows with a $\pi_1(M)$-equivariant homeomorphism between their respective orbit spaces $\orb$ and $\mathcal{O}_\psi$ with induced actions, then by construction the associated spaces $W^>_s(\flow)/\pi_1(M)$ and $W^>_s(\psi)/\pi_1(M)$ will be homeomorphic, via a homeomorphism respecting orbits of the induced 1-dimensional foliation and inducing the map $f_*$ on $\pi_1(M)$.  \qed

\subsection{General case} 

In the case where neither foliation admits an invariant orientation, the space $W^>_s$ (or its analogue $W^>_u$) does not inherit an action of $\pi_1(M)$.  However, this can be solved by a small modification to the definition.  For this, we use the existence of a {\em leafwise hyperbolic metric} on $M^3$.  
As there are no transverse invariant measures to the foliations of an Anosov flow, Candel's Uniformization Theorem (see, e.g., \cite[Section I.12.6]{CanCon}) implies that, 
for any Anosov flow, there exists a metric on $M$ such that its lift to the universal cover $\wt M$ has the property that all leaves of $\wt \cF^s$ are isometric to the hyperbolic plane. Moreover, the orbits are quasi-geodesics with a common forward endpoint, and pairwise distinct backwards endpoints (see e.g., \cite{BFP23}).  
 
Using such a choice of metric, we can define a canonical family of {\em leafwise involutions}.  For an orbit $\gamma$ of $\flow$, let $r_\gamma$ be the (isometric) reflection of $\hfs(\gamma)$ along the geodesic with endpoints shared by $\gamma$, and define an involution $i_\gamma$ on the set of orbits in $\hfs(\gamma)$ by sending an orbit $\alpha$ with negative endpoint $\zeta$ to the orbit with negative endpoint $i_\gamma(\zeta)$.  The fact that $\pi_1(M)$ acts on $\wt M$ by isometries means that this family of involutions is equivariant; precisely, for $g \in \pi_1(M)$ we have 
\begin{equation} \label{eq_Ws_equivariance}
 g i_\gamma(\alpha) = i_{g \gamma}(g \alpha)
 \end{equation}

Using this, we define the space 
\[
W_s :=\{ (\gamma, y, i_\gamma(y)) \in \orb \times \orb \times \orb \mid y \in \bfs(\gamma) \smallsetminus \gamma\}
\]
and now prove the following: 

\begin{theorem}[Reconstruction theorem] \label{thm_nonorientable_model_flow}
Let $\flow$ be an Anosov flow on $M$ and let $W_s$ be the space defined above, for some choice of leafwise hyperbolic metric. 
Let $\Psi$ be the 1-dimensional foliation of $W_s$ by constant first-coordinate leaves.  

Then $W_s$ has a natural coordinate-wise action of $\pi_1(M)$ which is properly discontinuous, free and cocompact, 
$W_s/\pi_1(M)$ is homeomorphic 
to $M$, and $\Psi$ descends to a $1$-dimensional foliation of $W_s/\pi_1(M)$ any flow-parametrization of which is orbit equivalent to $\flow$. 
\end{theorem}

\begin{proof} 
The proof is a simple adaptation of that for Theorem \ref{thm_transverse_orientable_model_flow}.
Let $\flow$ and $W_s$ be as above.  Forgetting the leafwise hyperbolic metric used to define $W_s$, we now choose a well-adapted metric for the flow, and define a map $h\colon W_s \to \wt M$ by defining 
$h(\gamma, y, i_\gamma(y))$ to be the unique point $p$ on the orbit $\gamma$ so that the distance along the strong stable leaf $E^{ss}(p)$ between the orbits $y$ and $i_\gamma(y)$ is exactly 1.  

Note that $h(\gamma, y, i_\gamma(y)) = h(\gamma, i_\gamma(y), y)$, and, for $g \in \pi_1(M)$, Equation \eqref{eq_Ws_equivariance} implies that 

$$g h(\gamma, y, i_\gamma(y)) = h(g \gamma, gy,  i_{g \gamma}(g y)).$$  

\noindent
As before, one can check directly that $h$ is a homeomorphism $W_s \to \wt M$, and so $W_s/\pi_1(M) \cong M$, and the constant first-coordinate foliation agrees with the orbit foliation of $\flow$. 
\end{proof} 

Barbot's theorem now follows using the same argument as in the transversally orientable case. 

\section{The nonsingular case: proof of Theorems \ref{thm_compact} and \ref{thm_noncompact}} \label{sec_nonsingular_main_theorem}

In this section, we will prove Theorem \ref{thm_compact} and its noncompact version (Theorem \ref{thm_noncompact}), i.e., deal with the case of a nonsingular bifoliated plane. The proofs in the singular case follow an identical strategy but are more notationally heavy, so for readability we treat the nonsingular case here and describe the necessary modifications to account for prongs in Section \ref{sec_pA_case}.

Recall that a group $G$ acts properly discontinuously on a 
space $Z$ if for any compact set $K \subset Z$ we have $g_n(K) \cap K 
\not = \emptyset$ for at most finitely many $g_n$ in $G$.
If $Z$ is first countable, this is equivalent to the following
condition: there is no sequence of points $z_n$ in $Z$
and $g_n$ pairwise distinct elements in $G$ such that
$z_n$ converges (to some point $z$ in $Z$) and $g_n(z_n)$
converges (to some point $v$ in $Z$).

For the rest of the section, we assume that $(P, \cF_1, \cF_2)$ is a nonsingular bifoliated plane without $\cF_1$-infinite product regions, and $W^>_1$ as defined in Definition \ref{def_W1}. We further suppose that $G<\Aut_1^+(P)$ 
is a torsion-free subgroup that acts properly discontinuously on  $W^>_1$. The two cases of $G$ being cocompact or not will be treated at the end.

\begin{lemma} \label{lem_torsion_free}
Under the assumptions above, $G$ acts freely on $W_1^>$.
\end{lemma}

\begin{proof} 
Let $(x, y) \in W_1^>$, and suppose $g$ fixes $(x,y)$. Then $K:=\{(x,y)\}$ is a compact set and $g^n(K) \cap K \neq \emptyset$ for all $n$.
 Since the action is properly discontinuous, this implies that $g$ is torsion and hence trivial. 
\end{proof}

As a consequence of Lemma \ref{lem_torsion_free}, $M: = W^>_1/G$ is a 3-manifold.  Furthermore, the sets $L_x := \{(x,t) \in W^>_1\}$ descend to form leaves of a 1-dimensional foliation of $M$.  
Let $\flow$ be any flow on $M$ with these leaves as orbits, and let $\hflow$ denote its lift to $\wt M = W^>_1$. 
By construction, the orbits of $\hflow$ are the sets $L_x$.  

Our goal is to show that $\flow$ is expansive. Since there are several equivalent characterizations of expansivity on \emph{compact} $3$-manifolds, but these may fail to be equivalent in the noncompact case, we state the definition we will take:
\begin{definition}\label{def_expansive}
We  say that a nonsingular flow $\psi$ on a 3-manifold $X$ is \emph{expansive} if the lifted flow $\wt\psi$ on the universal cover $\wt X$ has properly embedded orbits, and there exists a metric $d$ on $M$ and a constant $\delta>0$ such that the following property is satisfied: given $x, y \in \wt X$ if there exists a reparameterization $\tau$ with $\tau(0)=0$ such that  $\wt d(\wt\psi^t(x), \wt\psi^{\tau(t)}(y))< \delta$ for all $t$, then $x$ and $y$ are on the same orbit.
\end{definition} 

\begin{rem}
The definition above is not the standard definition of expansivity\footnote{In \cite{BW72}, a flow is called expansive if there exists a metric $d$ on $M$ satisfying that for all $\epsilon>0$ there exists $\delta>0$ such that, for any $x, y \in X$, if there exists a reparameterization $\tau$ with $\tau(0)=0$ such that $d(\psi^t(x), \psi^{\tau(t)}(y))< \delta$ for all $t$, then $x = \psi^s(y)$ where $|s|<\epsilon$.}, as introduced by \cite{BW72}. When $X$ is a compact 3-manifold, work of Inaba and Matsumoto \cite{IM90} and Paternain \cite{Pat93} implies they are equivalent.

The reason for not repeating the standard definition verbatim, is that when $X$ is noncompact, it  becomes dependent on the \emph{parametrization} of the flow, as one could take a flow satisfying \cite{BW72} definition and then slow it down outside of compact sets so that it takes arbitrarily long for the flow to leave small balls. If a flow is $C^1$ and satisfies the definition above, then there always exists a reparametrization that will satisfy the definition of \cite{BW72}. See, \cite{JNY20} where this is studied under the name of \emph{rescaled expansivity}. When $\flow$ is only assumed to be continuous, it seems likely that a flow satisfying our Definition \ref{def_expansive} would be at least orbit equivalent to one satisfying the definition of \cite{BW72}, but for us it is simpler to work with this topological version.  
\end{rem}

In order to prove that $\flow$ is expansive, we will cover $M$ by flow-boxes that are adapted to the description of $W^>_1$ as pairs of points in $P$. 

\begin{definition}\label{def_good_neighborhood}
A \emph{good neighborhood} of a point $(x,y)\in W^>_1$ is a subset $V_{(x,y)} = (R_x \times R_y) \cap W^>_1$ such that (see Figure \ref{fig_nbhd}):
\begin{enumerate}[label=(\roman*)]
\item $R_x$ and $R_y$ are neighborhoods of $x$ and $y$ (respectively) whose closures are each homeomorphic by a foliation-preserving homeomorphism to a rectangle $[0,1]^2$ with the trivial foliation. 
\item The saturations of $R_x$ and $R_y$ by $\cF_1$ leaves agree, and 
\item $\overline{R_x}\cap \overline{R_y} = \emptyset$.
\end{enumerate}
\end{definition}
 \begin{figure}
   \labellist 
  \small\hair 2pt
     \pinlabel $x$ at 50 21
     \pinlabel $y$ at 235 21
          \pinlabel $R_x$ at 50 0
     \pinlabel $R_y$ at 240 0
 \endlabellist
     \centerline{ \mbox{
\includegraphics[width=10cm]{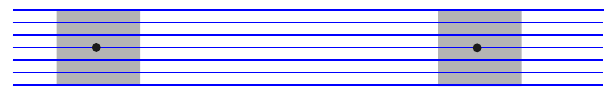}}}
\caption{A good neighborhood; $\cF_1$ is the horizontal foliation; the orientation of leaves is from left to right}
 \label{fig_nbhd} 
\end{figure}
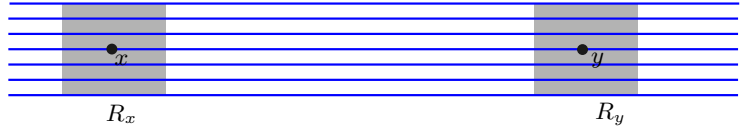

Note that, if $l$ is a $\cF_2$-leaf through $R_y$, then $(R_x \times l) \cap W^>_1 \subset V_{(x,y)}$ is a local section of $\hflow$.  Thus, good neighborhoods are flow boxes for $\hflow$, and any sufficiently small good neighborhood projects to a flow box for $\flow$ on $M$. 

Since our goal is to prove expansivity of the flow, and that this is a metric notion, in the case when $M$ is noncompact, some metrics may fail to see that expansivity. So our first goal is to build a good metric. Of course, when $M$ is assumed to be compact, this step is unnecessary. The key is to build a metric $d$ which admits a constant $0<\delta$ such that any ball of size $\delta$ in $M$ is contained in the projection of a good neighborhood of $W_1^>$.

Fix a countable, locally finite, cover $\mathcal U$ of $M$ by relatively compact flow boxes $U_i$, such that each box is the injective projection of a good neighborhood in $W^>_1$. Fix also a complete Riemannian metric $d_0$ on $M$ and an exhaustion by compact sets $K_n$ taken to be the closed balls of radius $n$ for $d_0$ about some point $x_0$.

If the Lebesgue number for $d_0$ of the cover $\mathcal U$ is positive, then we call $\delta_0$ that Lebesgue number, and we do not have to modify the metric. Otherwise, we will modify $d_0$ in order to obtain a metric with positive Lebesgue number.

Call $\delta(n)$ the Lebesgue number for $d_0$ of the cover $\mathcal U$ for the compact $K_n$. By definition, $\delta(n)$ is a non-increasing sequence. Then, inductively we define a metric $d_n$ by scaling the metric $d_{n-1}$ by a bump function which takes value of $\delta_0/\delta_n$ in $K_n\smallsetminus K_{n-1}$ and $1$ outside of a neighborhood of it.  Note that the Lebesgue number for $d_n$ of the cover $\mathcal U$ for the compact $K_n$ is now as close as we want to $\delta_0$ (with the closeness depending on our choice of bump function).
We call $d= d_{\infty}$ the metric obtained by running this process on all $n$.

By construction, the Lebesgue number for $d$ of the cover $\mathcal U$ on the whole of $M$ is positive, and we let $\delta_0$ denote this number.  

\begin{rem}\label{rem_unbounded_diameter}
The metric $d$ we built has the property that for any point $p$ in $M$ there exists a flow box containing the ball of radius $\delta_0$ around $p$. By construction it even has the stronger property that such a flow box can be chosen amongst the ones in $\mathcal U$. But what we may have lost in the construction of $d$ is a control of the diameters of the elements of $\mathcal U$, i.e., when $M$ is noncompact, there may exists elements $U_i$ of arbitrarily large diameter for $d$. 
\end{rem}

Let $\{V_j\}$ be the cover of $\wt M = W^>_1$ consisting of good neighborhoods that are lifts of the elements $U_i$ of the cover defined above.  We denote the lifted metric on $\wt M$ by $\wt d$.  

We prove the following proposition. 

\begin{proposition} \label{prop_expansive_1}
Let $p, q \in M$ and suppose there exists a reparameterization $\tau$ with $\tau(0)=0$ such that $d(\flow^t(p), \flow^{\tau(t)}(q)) < \delta_0$ for all $t$.   Then $p$ and $q$ lie in a common flow box $U_i$ and on the same local orbit of $\flow$ in $U_i$. 
In particular, $\flow$ is expansive.
\end{proposition}

\begin{proof}[Proof of Proposition \ref{prop_expansive_1}]
Suppose $p, q$ and $\tau$ are as above and choose lifts $\wt p, \wt q$ such that 
\[ \wt d(\hflow^t(\wt p), \hflow^{\tau(t)}(\wt q)) < \delta \text{ for all } t. \]
Then for all $t$, there exists a flow box $V_t$ such that $\{\hflow^t(\wt p), \hflow^{\tau(t)}(\wt q)\} \subset V_t$.  Recall $V_t$ has the form $V_t = (R_t \times Q_t) \cap W^>_1$ where $R_t, Q_t$ are rectangles in $P$. 

Using the structure of $W^>_1$, we can write $\hflow^t(\wt p) = (x, y_t) \in W^>_1$, and  $\hflow^{\tau(t)}(\wt q) = (z, w_t)$.  
In any (bi)-foliated plane, all leaves are necessarily properly embedded.  As a consequence, 
up to reversing the direction of the flow, we assume that $y_t$ leaves all compact sets of $P$ as $t \to \infty$. 

First we show that $x$ and $z$ are on the same $\cF_2$-leaf: Suppose for a contradiction that $z\notin\cF_2(x)$. Let $u = \cF_1(x)\cap \cF_2(z)$, and up to switching the roles of $x$ and $z$, we can assume without loss of generality that $u\in \cF_1(x)^>$. Then, for $t$ sufficiently negative, we can assume that $y_t$ is between $x$ and $u$. But, for such a choice of $t$, the rectangle $Q_t$ which contains both $y_t$ and $w_t$ will have to also contain $z$, so in particular would intersect $R_t$, contradicting the definition of good neighborhood. See Figure \ref{fig_contradiction}.
 \begin{figure}[h]
   \labellist 
  \small\hair 2pt
     \pinlabel $x$ at 43 10
     \pinlabel $y_t$ at 95 10
          \pinlabel $z$ at 170 55
     \pinlabel $w_t$ at 225 55
 \endlabellist
     \centerline{ \mbox{
\includegraphics[width=9cm]{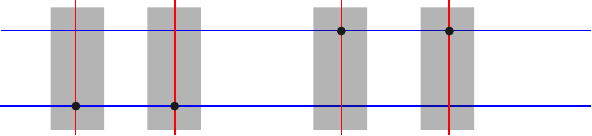}}}
\caption{If $y_t$ is close to $x$, then $(x, y_t)$ cannot share a good neighborhood with $(z, w_t)$}
 \label{fig_contradiction} 
\end{figure}
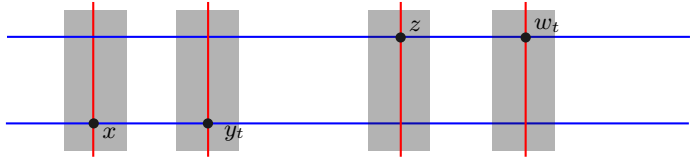

Thus, we have that $z\in \cF_2(x)$.
Next, we want to deduce that $z=x$. Suppose for a contradiction that $x\neq z$. In particular $\cF_1(z)\neq \cF_1(x)$. Now pick any point $u$ in the $\cF_2$-segment between $x$ and $z$.
By construction, the orbits $\flow^t(p)$ and $\flow^{\tau(t)}(q)$ in $M$ corresponding to the projections of $(x,y_t)$ and $(z,w_t)$ are in the same flow box, and one the same $\cF_2$-leaf. Moreover, they are on the same \emph{local} leaf of the projection of $\cF_2$: Letting $u$ vary from $x$ to $z$ on $\cF_2(x)$ and choosing $v_t\in \cF_1(u)\cap Q_t$ gives a continuous path of orbits from $\flow^t(p)$ to $\flow^{\tau(t)}(q)$ staying in that same flow box and on that same $\cF_2$-leaf.

Let $\pi\colon W^>_1 \to M$ be the projection. 
We show that $v_t$ escapes compact sets:\footnote{If we knew that the flow boxes had uniformly bounded diameter, then this would be automatic, but by Remark \ref{rem_unbounded_diameter}, this may fail for our choice of metric.}
Otherwise $\pi((u,v_t))$ stays in a compact set of $M$, and $\pi(V_t)$ 
intersects a compact set, therefore there are finitely many possibilities for $\pi(V_t)\in \mathcal U$, as $\mathcal U$ is locally finite. Hence there are finitely many possibilities for $V_t$ also. This contradicts that $y_t$ escapes compact sets.

Since $y_t$, $w_t$ and $v_t$ must all escape compact sets in $P$, and that $u$ was chosen arbitrarily between $x$ and $z$, we deduce that the region bounded by $\cF_1(x)$, $\cF_1(z)$ and the $\cF_2$-segment between $x$ and $z$ is an infinite product region. See Figure \ref{fig_product_region_contradiction}. This contradicts our assumption.  

 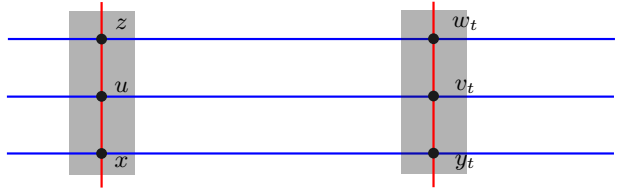
\begin{figure}[h]
   \labellist 
  \small\hair 2pt
     \pinlabel $x$ at 45 10
     \pinlabel $y_t$ at 180 10
     \pinlabel $u$ at 45 40
          \pinlabel $v_t$ at 180 40
          \pinlabel $z$ at 45 65
     \pinlabel $w_t$ at 180 65
 \endlabellist
     \centerline{ \mbox{
\includegraphics[width=8cm]{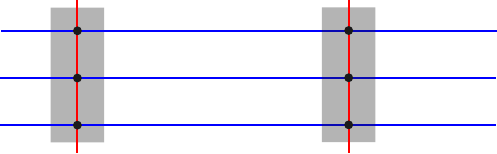}}}
\caption{If $x$ and $z$ are on the same $\cF_2$-leaf, their orbits cannot stay forever in the same flow box.}
 \label{fig_product_region_contradiction} 
\end{figure}

Therefore we deduce that $x=z$, which shows that $p, q$ lie on the same local orbit of the flow. Since orbits of $\wt \flow$ are properly embedded by construction, we deduce that $\flow$ satisfies Definition \ref{def_expansive}.
\end{proof}

With this we can easily finish the proof of Theorems \ref{thm_compact} and \ref{thm_noncompact}.   
By construction, since $\wt M = W^>_1$ and the foliation by orbits is the constant first-coordinate foliation, $P$ is the space of orbits of $\hflow$ in $\wt M$.  The one-dimensional foliations $\cF_1$ and $\cF_2$ on $P$, and thus their product with $\bR$ giving two-dimensional foliations on $W^>_1$, are invariant under the action of $G = \pi_1(M)$, and so give invariant foliations for $\flow$, with leaves formed by unions of orbits. Moreover, in the proof above, we saw that two orbits $(x,y_t)$ and $(z,w_t)$ can be parametrized to stay in common good neighborhoods in the future if and only if $x$ and $z$ are in the same $\cF_1$-leaf, and conversely, they stay close in the past if and only if they are in the same $\cF_2$-leaf. Thus, $\cF_1$ must project to the stable foliation of the flow $\flow$ and $\cF_2$ must project to the unstable foliation.  

Therefore, if $M=W^>_1/G$ is not compact, we have obtained the conclusion of Theorem \ref{thm_noncompact}. When $M$ is compact, the conclusion of Theorem \ref{thm_compact} follows by the work of Inaba and Matsumoto \cite{IM90} and Paternain \cite{Pat93}: Indeed, they showed that any expansive flow without fixed points (as is our case here) on a compact 3-manifold are pseudo-Anosov. 
\qed

\begin{rem}\label{rem_expansive_vs_pA_noncompact}
In the case of a noncompact manifold, the flow $\flow$ we obtain is expansive and leaves invariant two transverse foliations saturated by orbits. The one condition in the definition of a topological Anosov flow that $\flow$ may not necessarily enjoy is that the distance between two orbits that escape all compacts on the same $\cF_1$-leaf actually goes to $0$ in the future. It seems reasonable to expect that one could further modify the metric $d$ to satisfy this additional condition, but we did not try to verify this.
\end{rem}

\section{The closing property and proper discontinuity}  \label{sec_proper_discontinuity}

In this section, we will show Theorems \ref{thm_closing_implies_prop_disc} and \ref{thm_closing_plus_transitive_implies_prop_disc}.  

As an easy warm-up, we show that all points of $W^>_1$ are wandering (Lemma \ref{lem_wandering_2}).  This property is often mistakenly confused with the (stronger) property of the action being properly discontinuous. See \cite{Kapovich} for a detailed discussion.   These lemmas will also be useful for the proof of proper discontinuity. 

For convenience we introduce the following terminology.  
\begin{definition} A {\em closing pair} is a pair of open sets $R, U$ of $P$, with $R \subset U$ of $P$, such that if $gR \cap R \neq \emptyset$ then $g$ has a fixed point in $U$.  
For a point $(x, x') \in W^>_1$, a {\em good closing pair} for $(x,x')$ is a pair of good neighborhoods $(R \times R') \cap W^>_1$ and $(U \times U') \cap W^>_1$ of $(x,x')$, such that $R, U$ and $R', U'$ are both closing pairs.  
\end{definition} 

We also need the following elementary observation. 
\begin{observation} \label{obs_empty_intersection}
If $x \not = y \in P$ are hyperbolic fixed points of $g$, then $\cF_1(x) \cap \cF_2(y) = \emptyset$.  \footnote{Recall our convention is that a fixed point $x$ is hyperbolic if (up to switching $g$ with $g^{-1}$), $g$ is topologically expanding on $\cF_1(x)$ and topologically contracting on $\cF_2(x)$.  }

\end{observation}
\begin{proof}
If  $\cF_1(x) \cap \cF_2(y) \neq \emptyset$, then this intersection is a single point, say $z$, which is distinct from both $x$ and $y$; and $z$ is fixed by the power of $g$ that preserves all the rays of $\cF_1(x)$ and $\cF_2(y) $.  Thus, some nontrivial power of $g$ fixes multiple points on the same leaf, contradicting hyperbolicity.   
\end{proof}

\begin{lemma}[Points are wandering] \label{lem_wandering_1}
 Let $(p,p') \in W^>_1$, and let $(R \times R') \cap W^>_1$ and $(U \times U') \cap W^>_1$ be a good closing pair for $(p,p')$.   If $g R \cap R \neq \emptyset$ and $g R' \cap R' \not = \emptyset$ then $g = \id$.  
\end{lemma} 

\begin{proof}
From the closing property, if $g R \cap R \neq \emptyset$ and $g R' \cap R' \not = \emptyset$ then $g$ has fixed points $z \in U$ and $z' \in U'$.  By definition of good neighborhood we have $\cF_1(z) \cap \cF_2(z') \neq \emptyset$, so Observation \ref{obs_empty_intersection} implies $g = \id$.  
\end{proof} 

To rephrase this in the standard language for ``wandering", given $(p,p')$ as above, take $Z=(R \times R') \cap W^>_1$ to be the neighborhood of $(p,p')$ from Lemma \ref{lem_wandering_1}.  Then if $g \neq id$, we have $g(Z) \cap Z = \emptyset$, so $(p,p')$ is a wandering point.

\begin{remark}
Note that the above proof that the action is wandering did not use any assumptions other than that the fixed points are hyperbolic and the closing property.  These assumptions are satisfied by many actions.
\end{remark} 

Rephrasing the wandering condition gives the following useful lemma: 
\begin{lemma} \label{lem_wandering_2}
Assume $G < \Aut(P)$ has the closing property.  If for some $(p,p')$ and $(q, q') \in W^>_1$ 
there exist $g_n$ such that $g_n(p) \to q$ and $g_n(p') \to q'$, then the sequence $g_n$ is eventually constant. 
\end{lemma}

\begin{proof}
 Take a good neighborhood $(R \cup R') \cap W^>_1$ of $(q, q')$ in $W^>_1$ as in  Lemma \ref{lem_wandering_1}.  For sufficiently large $n, m$ we have $g_n g_m^{-1}(R) \cap R \neq \emptyset$: for $n, m$ large, then $g_n(p), g_m(p)
\in R$, so 

$$g_n g_m^{-1}(g_m(p)) \ = \ g_n(p) \ \ {\rm is \ in} \ \ 
R \cap g_n g^{-1}_m(R),$$

\noindent
and in the same way
$g_n g_m^{-1}(R') \cap R' \neq \emptyset$.  By Lemma \ref{lem_wandering_1}, we conclude that $g_n = g_m$.  
\end{proof} 

Our main goal is to prove the following. 
\begin{proposition}[Closing property plus uniformly hyperbolic gives proper discontinuity]\label{prop_proper_discontinuity}
Suppose $G$ acts on $P$ with the closing property and
uniformly hyperbolic fixed points, and suppose there exists a convergent sequence 
 $(x_n, x'_n) \to (x, x')$ in $W^>_1$, and $g_n \in G$ such that $g_n(x_n, x'_n) \to (y, y') \in W^>_1$.  Then the sequence $g_n$ is eventually constant. 
\end{proposition} 

For this we need one more preliminary result. 

\begin{lemma} \label{lem_fixed_points_shrink}
Suppose that $G$ has uniformly hyperbolic fixed points, $(x_n, x'_n) \to (x, x')$ in $W^>_1$, and  $g_n(x_n, x'_n) \to (y, y') \in W^>_1$.  
Let $V$ be a compact rectangle in $P$ containing $x$ and $x'$ in its interior. Assume that for all $m<n$, $g_m^{-1}g_n \neq \id$   and has a fixed point in $V$. 
Then
$\cF_1(g_nV)$ converges to the leaf $\cF_1(y)$.
\end{lemma}

\begin{proof}
Since fixed points are hyperbolic, for each pair $m<n$, the intersection of $V$ and $g_m^{-1}g_n V$ is ``Markovian'', i.e., we have exactly one of the following two possibilities:
\begin{enumerate}
\item\label{item_inproof_case1} $\cF_1(g_mV) \supset \cF_1(g_nV)$ and $\cF_2(g_mV) \subset \cF_2(g_nV)$, or 
\item \label{item_inproof_case2} $\cF_2(g_mV) \supset \cF_2(g_nV)$ and $\cF_1(g_mV) \subset \cF_1(g_nV)$. 
\end{enumerate}
Moreover,  these containments are strict whenever $n \neq m$, since $g_m^{-1}g_n \neq \id$ (by our assumption) and its fixed points are hyperbolic.

\begin{claim} 
There is no infinite subsequence such that case \ref{item_inproof_case2} holds for all $m<n$ along the subsequence.  
\end{claim} 
\begin{proof} 
Suppose  for contradiction that $g_{n_k}$ is a subsequence where case \ref{item_inproof_case2} holds, so $\cF_2(g_{n_k}V)$ is decreasing as $k \to \infty$.  We first show this leads to a contradiction.  Relabel $g_{n_k}$ by $g_k$.  Fix some 
some $k_0$, let $R=g_{k_0}V$ and, for $n>k_0$ let $h_n= g_ng_{k_0}^{-1}$. Then, for all $n>k_0$, $h_n$ has a fixed point in $R$ (it is the image by $g_{m_0}$ of $p_{n,m_0}$ the fixed point of 
$g_{k_0}^{-1}g_{n}$), and $\cF_2(h_n R) = \cF_2(g_n V)$ is a decreasing sequence.  
By assumption the action of $G$ has uniformly hyperbolic fixed points, so $\cF_2(h_n R) = \cF_2(g_n V)$ must converge to a single leaf of $\cF_2$. 
However, as $g_n(x_n) \to y, g_n(x'_n) \to y'$ and $y' \not \in
\cF_2(y)$, we deduce that for all sufficiently large $n$, $g_nV$ contains both $g_n x_n$ and $g_n x_n'$. 
This contradicts the fact that $\cF_2(g_nV)$ must converge to a single leaf, and proves the claim. 
\end{proof} 

Now we can finish the proof of the Lemma.  Suppose for contradiction that 
$\cF_1(g_nV)$ does not converge to a segment of $\cF_1(y)$.  Since it contains the segment $g_n(\cF_1(x) \cap V)$ which converges to a segment of $\cF_1(y)$, we must have some subsequence $g_{n_k}$ such that $g_{n_k}(V)$ contains a nondegenerate rectangle $U$.  
By the claim above, $\cF_2(g_{n_k}(V))$ is not decreasing, so it must be the case that $\cF_1(g_{n_k}(V))$ is decreasing.  In this case, the uniform hyperbolicity of fixed points imply that $\cF_1(g_nV)$ must converge to a single leaf, contradicting that it must contain a nondegenerate segment. 
\end{proof} 

Finally, we record for use in the proof a very elementary lemma.  
\begin{lemma}\label{lem_disjoint_or_not}
 Let $B_n, n \in \mathbb{Z}$ be a collection of
nonempty subsets of some space $X$. Then there is an infinite sub-collection, denoted by 
$C_k = B_{n_k}$ satisfying: either 
\begin{enumerate}[label=(\roman*)]
\item \label{item_nonempty_intersection} $C_k \cap C_i \not = \emptyset$ for
all $k \not = i$, or 
\item \label{item_empty_intersection} $C_k \cap C_i = \emptyset$ for
all $k \not = i$
\end{enumerate} 
\end{lemma}

\begin{proof}
Suppose there is $n_0$ such that the cardinality $\# \{ n \ | \ B_n \cap B_{n_0} \not = \emptyset \}$
is infinite. If that is the case, choose $n_0$ to be the smallest
such one, and let $C_1 = B_{n_0}$. By hypothesis, we can then eliminate all the elements
of the sequence $B_n$ which do not intersect $B_{n_0}$ as well
as those $B_n$ with $n < n_0$, and obtain a subsequence. We now iterate that argument:
of the remaining elements if there is one which intersects
infinitely many others, we call $C_2$ the first such element, 
and eliminate all elements that do not intersect $C_2$ as well as
those elements that precede $C_2$. 

If this can be done forever, then by construction the sequence
$C_k$ is a subsequence of the original sequence and satisfies
that $C_k \cap C_i \not = \emptyset$ for all $k \not = i$.
In this case we obtain option \ref{item_nonempty_intersection} of the claim.

Otherwise at some point we cannot continue. This means that there
is a subsequence $B_{n_i}$ of the original sequence
so that each $B_{n_i}$ intersects only finitely many others.  Let $C_1 = B_{n_1}$, and 
discard all (finitely many) $B_{n_i}$  which intersect $C_1$.  
Let $C_2$ be the first of the remaining sets, and restart the process.  
This produces an infinite sequence $C_k$ such that for any $k \not = i$ we have $C_k \cap C_i
= \emptyset$. This is case \ref{item_empty_intersection}.
\end{proof}

\begin{proof}[Proof of Proposition \ref{prop_proper_discontinuity}]
Suppose  $(x_n, x'_n) \to (x, x')$ and $g_n(x_n, x'_n) \to (y, y')$ in $W^>_1$.   We assume for contradiction that the sequence $g_n$ is not eventually constant. So, after passing to a subsequence we may assume $g_n \neq g_m$ for all $n \neq m$. 

Let $(R_x \times R'_x) \cap W^>_1 \subset (U_x \times U'_x) \cap W^>_1$ be a good closing pair for $(x,x')$.  By Lemma \ref{lem_wandering_1}, if $g_m^{-1}g_n R'_x \cap R'_x \neq \emptyset$ then $g_m^{-1}g_n R_x \cap R_x = \emptyset$.  Thus, up to switching the labels $x, x'$ (and reversing orientation on the $\cF_1$ leaves), we can assume that after passing to a (further) subsequence we have $g_m^{-1}g_n R_x \cap R_x = \emptyset$ for all $m \neq n$.  Equivalently, 
\begin{equation} \label{eq_disjoint} 
g_n R_x \cap g_m R_x =  \emptyset \text{ for all } m \neq n.
\end{equation}

Since the rectangles $g_n R_x$ contain points which converge to $y$, Equation \eqref{eq_disjoint} implies that the projection of $g_n R_x$ to the leaf space of at least one of $\cF_1$ or $\cF_2$ shrinks to a point or union of nonseparated points; that is we have either 
\begin{enumerate}[label=(\arabic*)] 
\item\label{item_case_F_1_shrinks}  Up to a subsequence, the projection of $g_n R_x$ to the $\cF_1$ leaf space limits to $\cF_1(y)$ or a union of nonseparated leaves containing $\cF_1(y)$, or
\item\label{item_case_F_2_shrinks} Up to a subsequence, the projection of $g_n R_x$ to the $\cF_2$ leaf space limits to $\cF_2(y)$ or a union of nonseparated leaves containing $\cF_2(y)$. 
\end{enumerate} 

Our next goal is to reduce to case \ref{item_case_F_1_shrinks}.  

First, apply Lemma \ref{lem_disjoint_or_not} to pass to a subsequence such that 
 $g_m^{-1} R'_x \cap R'_x \not = \emptyset$ for
all $m \not = n$, or $g_m^{-1}g_n R'_x \cap R'_x = \emptyset$ for all $n \neq m$. 

The next claim shows that the first situation gives case \ref{item_case_F_1_shrinks}:
 
\begin{claim} \label{case_1_holds} 
Assuming that $g_m^{-1}g_n R_x \cap R_x = \emptyset$, 
if $g_m^{-1}g_n R'_x \cap R'_x \neq \emptyset$ for all $m,n$, then \ref{item_case_F_1_shrinks} holds.
\end{claim} 
\begin{proof} 
Suppose $g_m^{-1}g_n R'_x \cap R'_x \neq \emptyset$ for all $m, n$.  Then by the closing property, $g_m^{-1}g_n$ has a fixed point in $U'_x$. 
Let $V$ be the rectangle consisting of the union of $U_x$, $U'_x$ and the segment of each leaf of $\cF_1(R_x)$ between $U_x$ and $U'_x$.

Since $g_m^{-1}g_n$ has a fixed point in $V$ for all $m,n$, Lemma \ref{lem_fixed_points_shrink} implies the claim.
\end{proof} 

Now, if we are not in this situation and
case \ref{item_case_F_1_shrinks} does not hold, we next show we can swap the roles of $x$ and $y$ and obtain an equivalent situation in which case \ref{item_case_F_1_shrinks} holds, as follows:  Consider the sequences $y_n := g_n(x_n) \to y$ and $y'_n := g_n(x'_n) \to y'$, and note that $h_n(y_n) \to x$ and $h_n(y'_n) \to x'$ where $h_n = g_n^{-1}$.   Let $(R_y \times R'_y) \cap W^>_1 \subset (U_y \times U'_y) \cap W^>_1$ be a closing pair for $(y,y')$ in $W_1^>$.

We can apply the same analysis we did previously for $(x,x')$ to $(y,y')$.
So as in that analysis we can assume, up to switching $y, y'$ and
reversing the orientation of $\cF_1$, that
$h_m^{-1}h_n R_y \cap R_y = \emptyset$ for every $m \not = n$. 
Since case \ref{item_case_F_2_shrinks} holds for the sequence $g_n$, we have $\cF_2(g_n R_x) \subset \cF_2(R_y)$ 
for all $n$ sufficiently large.  Thus, $\cF_2(h_nR_y) \supset \cF_2(R_x)$ for all $n$ sufficiently large.  In particular, \ref{item_case_F_2_shrinks} \emph{cannot} hold for the $h_n$, and we have a parallel setting for 
$y_n, y'_n, h_n$ where \ref{item_case_F_1_shrinks} holds.   Thus, we have arrived at the reduction to case  \ref{item_case_F_1_shrinks}, and it remains simply to derive a contradiction under this assumption.  

\begin{claim} 
Let $J \subset \cF_1(x)$ be the segment between $R_x$ and $R'_x$. Then $g_n(J)$ converges to a single point. 
\end{claim} 

\begin{proof}
Let $a, a' \in \cF_1(x)$ be the endpoints of $J$.  Since $g_n(x) \to y, g_n(x') \to y'$, it follows that 
$$\cF_2(g_n(x)) \cap \cF_1(y) \ \ \ {\rm and} \ \ \ 
\cF_2(g_n(x')) \cap \cF_1(y)$$
\noindent
stay in a compact interval of $\cF_1(y)$. If $g_n(J)$ does not limit to a single point, then we can pass to a subsequence so that $g_n(a) \to b$ and $g_n(a') \to b'$ for some $b \neq b'$. Moreover, we must have that both $b$ and $b'$ are on the closed interval of the $\cF_1$-leaf between $y$ and $y'$.  Thus $(b,b') \in W^>_1$, contradicting Lemma \ref{lem_wandering_2}. 
\end{proof} 

Let $z = \lim_{n \to \infty} g_n(J)$.  
By construction, $z$ lies on the closed interval between $y$ and $y'$ in $\cF_1(y)$.
Suppose $V_z \subset U_z$ is a pair of neighborhoods of $z$ where the closing lemma applies.  Then for any pair $m, n$ large enough, the map $g_m g_n^{-1}$ takes $g_n(J) \subset V_z$ to $g_m(J) \subset V_z$, so 
has a fixed point in $U_z$.    We may choose such a $U_z$ small enough so that it can be extended to a compact rectangle $U$ with $y$ and $y'$ in its interior; thus $U$ has the property that $g_m g_n^{-1}$ has a fixed point in $U$ for all sufficiently large $m, n$.   Apply Lemma \ref{lem_fixed_points_shrink}, with $(y,y')$ playing the role of $(x,x')$ and $g_n^{-1}$ that of $g_n$.  The conclusion states that $\cF_1(g_n^{-1} U)$ converges to the leaf $\cF_1(x)$.

In particular, we must have that for all large enough $n$, $\cF_1(g_n^{-1} U)\subset \cF_1(R_x)$, or equivalently $\cF_1(U)\subset \cF_1(g_n R_x)$.  But this contradicts the assumption that we were in case \ref{item_case_F_1_shrinks}, i.e., that $\cF_1(g_n R_x)$ must shrink to a single leaf or union of nonseparated leaves. This contradiction proves the proposition.  
\end{proof}

Given Proposition \ref{prop_proper_discontinuity}, Theorem \ref{thm_closing_implies_prop_disc} is almost immediate:

\begin{proof}[Proof of Theorem \ref{thm_closing_implies_prop_disc}]
By Proposition \ref{prop_proper_discontinuity}, the action of $G$ on $W_1^>$ is properly discontinuous.  To see that it is free, we apply an argument as in Lemma \ref{lem_wandering_1}:  if a nontrivial $g\in G$ fixes a point $(x,x')\in W_1^>$, then in particular it fixes both $x$ and $x'$ in $\cF_1(x)$.
This contradicts the property of hyperbolic fixed points.
\end{proof}

\subsection{Proof of Theorem \ref{thm_closing_plus_transitive_implies_prop_disc}}
Given the result of Proposition \ref{prop_proper_discontinuity}, Theorem \ref{thm_closing_plus_transitive_implies_prop_disc}, 
is an immediate consequence of the following: 

\begin{theorem}[Conditions for uniformly hyperbolic fixed points]\label{thm_transitivity_and_discrete_implies_uniform}
Let $G < \Aut^+_1(P)$.  Assume
\begin{enumerate}[label=(\roman*)]
\item $G$ has the closing property,
\item $G$ has hyperbolic fixed points,
\item For any $x\in P$ fixed by some nontrivial $g\in G$, the orbit $G\cdot x$ is closed and discrete, and
\item The set of fixed points of nontrivial elements of $G$ is dense in $P$.
\end{enumerate}
Then $G$ acts with uniformly hyperbolic fixed points.
\end{theorem}

We start with a lemma that may be of independent interest. Its proof is an adaptation to our setting of that of \cite[Proposition 3.1.1]{BM_book}.
\begin{lemma}\label{lem_cyclic_stabilizer}
Let $G < \Aut^+(P)$. Suppose that 
\begin{enumerate}[label=(\roman*)]
\item $G$ has the closing property, and
\item $G$ has hyperbolic fixed points.
\end{enumerate}

Then, for any $x\in P$, its stabilizer $\mathrm{Stab}_G(x)$ in $G$ is either trivial or infinite cyclic. 
\end{lemma}

\begin{proof}
Let $x \in P$, and let $\mathrm{Stab}_G(x)$ denote its stabilizer in $G < \Aut^+(P)$.
Since $G$ acts with hyperbolic fixed points, $x$ is the unique fixed point in $\cF_1(x)$ of any nontrivial element of $\mathrm{Stab}_G(x)$.  Furthermore, the action of $\mathrm{Stab}_G(x)$ restricted to any ray of $\cF_1(x)$ (say for instance $\cF_1^>(x)$) is faithful and free.  Thus, we may apply H\"older's Theorem (see, e.g., \cite[Section 2.2.4]{Nav11}), and conclude that $\mathrm{Stab}_G(x)$ is abelian and the action of $\mathrm{Stab}_G(x)$ on this ray is semi-conjugate to an action by translations on $\bR$.  If $\mathrm{Stab}_G(x)$ is not cyclic, then this action is semi-conjugate to an indiscrete group of translations. This means that there exists $x'$ in $\cF_1^>(x)$ and a sequence $h_n\in\mathrm{Stab}_G(x)$ such that $h_n x'$ accumulates on $x'$. But, since $h_n x = x$ for all $n$, this implies that $h_n(x,x')$ accumulates to $(x,x')$, contradicting the fact that the action of $G$ on $W_1^>$ is wandering (Lemma \ref{lem_wandering_2}). 
\end{proof}

\begin{proof}[Proof of Theorem \ref{thm_transitivity_and_discrete_implies_uniform}]
Suppose for a contradiction that there exists a compact rectangle $R$ and a sequence $g_n\in G$ of distinct elements such that $g_n$ fixes a point in $R$ and $\cF_2(g_nR)$ is decreasing but does not converge to a single leaf.  (The case for $\cF_1$ in place of $\cF_2$ is identical).  Since fixed points are hyperbolic, we have the sequence $\cF_1(g_nR)$ must be increasing, i.e., $\cF_1(g_{n-1}R) \subset \cF_1(g_{n}R)$. 
This monotonicity implies that $g_nR$ limits to a nondegenerate trivially foliated region $U$ in $P$.  

Since $U$ is not contained in a single leaf, it has non-empty interior. By density of fixed points, there exists $x\in \mathring U$ fixed by some nontrivial element $h\in G$. Since $x\in \mathring U$, for all $n$ large enough, we must have that $x\in g_nR$, which implies that, for all $n$ large enough, $g_n^{-1}x\in R$.

Since $G\cdot x$ is closed and discrete, this implies that, up to a subsequence, $g_n^{-1}x$ is eventually constant. So without loss of generality, we may assume that for all $n,m$, $g_n^{-1} x = g_m^{-1} x$.
In particular, for all $n,m$, we have $g_n g_m^{-1} \in \mathrm{Stab}_G(x)$. 

By Lemma \ref{lem_cyclic_stabilizer}, $\mathrm{Stab}_G(x)$ is cyclic.  Let $h$ denote the generator of $\mathrm{Stab}_G(x)$ that acts as a contraction on $\cF_2(x)$. Then, we have that for all $n>1$, there exists $k_n$ such that $g_n g_1^{-1} = h^{k_n}$ and
\[
\cF_1(g_n R)= \cF_1(g_n g^{-1}_1 (g_1R)) = \cF_1(h^{k_n} g_1R))
\]
Since $\cF_1(g_n R)$ is assumed to be a decreasing sequence, and $x\in g_1 R$, we must have $k_n \to +\infty$ as $n \to \infty$.  Since $h$ acts as a contraction on $\cF_2(x)$, we therefore deduce that $\cF_1(h^{k_n} g_1R)= \cF_1(g_n R)$ must actually converge to $\cF_1(x)$, which contradicts our original assumption.
\end{proof}

%%%%%%%%%%%%%%%%%%%%

\section{General case: planes with singularities} \label{sec_pA_case}
We describe the necessary modifications and adaptations to prove our results on bifoliated planes with prong singularities, 
starting with Theorems \ref{thm_compact} and \ref{thm_noncompact}.  

We first need to introduce the space $W^\ast_1$ and modify the definition of good neighborhoods, to account for the fact that the neighborhood of a prong singularity in $P$ is not a trivially foliated rectangle but rather a semi-branched cover of such.  As explained in the introduction, the general idea is as follows: in the original definition of $W^>_1$, the second coordinate (moving along a leaf) parametrized the orbits of a flow.  When this coordinate lies on a singular leaf, we still wish for it to live in a 1-parameter family, so we make an ad-hoc identification of the prongs.  

To make this precise, let $(P, \cF_1, \cF_2)$ be a bifoliated plane, with only even prong singularities (at most one on each leaf), and without $\cF_1$-infinite product regions. 
Suppose $G$ is a torsion-free group acting by automorphisms of $P$, preserving orientation along leaves of $\cF_1$.  As before, we denote the positive side of $\cF_1(x)$ by $\cF_1^>(x)$. 
For the set-up, we assume also that the stabilizer of any prong singularity under the action of $G$ is discrete.  This hypothesis is satisfied by any pseudo-Anosov flow (and also follows from the closing property and hyperbolic fixed points as in Lemma \ref{lem_cyclic_stabilizer}).

For each $2k$ prong singularity $p \in P$, let $r_1(p), \ldots, r_k(p)$ denote the rays in $\cF_1^>(p)$  of the prong, and fix a proper homeomorphism $\sigma^p_j\colon r_1(p) \to r_j(p)$ that is equivariant with respect to $\Stab_G(p)$ in the following sense:  if $g \in \Stab_G(p)$ fixes all rays through $p$, then $\sigma^p_j(g(x)) = g \sigma^p_j(x)$.   One can define such homeomorphisms arbitrarily on a fundamental domain for the (cyclic) stabilizer of $r_1$ and then extend equivariantly; fixing $\sigma_1 = id$. 

\begin{definition}\label{def_W_singular}
Define a space $W_1^\ast \subset P \times 2^P$ as follows. 
We say $(x, Y) \in W_1^\ast$ if: 
\begin{itemize}
\item $Y = \{y\}, \, y \in \cF_1^>(x)$, $x$ is not singular, and $\cF_1^>(x)$ has no singular point between $x$ and $y$, 
or
\item $Y = \{y, \sigma^p_2(y), \ldots, \sigma^p_k(y) \}$, $y \in \cF_1^>(x)$  and  $p \in \cF_1(x)$ is a 2k-prong, either between $x$ and $y$  or equal to $x$. 
\end{itemize} 
We topologize $W_1^\ast$ by saying that $(x_n, Y_n)$ converges to $(x, Y)$ if $x_n \to x$, and {\em some point} of $Y_n$ converges to some point of $Y$.  
\end{definition}

One can verify from the definition that $W_1^\ast$ is homeomorphic to $P \times \bR = \bR^3$.  Note that $P$ was assumed to have a discrete set of prongs, and at most one prong on any leaf.  

With this setup, the proof of Theorem \ref{thm_compact} goes through once we have defined appropriate generalization of the good neighborhoods of Definition \ref{def_good_neighborhood}.

If $\cF_1(x)$ is nonsingular, then a good neighborhood for $(x,\{y\})$ is defined as before. 
When $\cF_1(x)$ is singular, there are four cases to treat, and we describe quickly how to replace the rectangles $R$ and $Q$:
\begin{enumerate}
\item $(x,Y)= (x,\{y\})$ and neither $x$ nor $y$ are singular points. Then $R$ and $Q$ are both rectangles around $x$ and $y$, as in the nonsingular case.
\item $(x,Y)= (x,\{y\})$ and $y$ is singular. Then $R$ is a rectangle containing $x$ and $Q$ is a union of two closed rectangles containing $y$ in their shared boundary. 
\item $(x,Y) = (x, \{y, \sigma^p_2(y), \ldots, \sigma^p_k(y) \})$ and $x$ is nonsingular. Then $R$ is a rectangle containing $x$ and, assuming that, up to renaming the elements in $Y$, $y$ and $\sigma^p_2(y)$ are the two elements of $Y$ on the faces of $\cF_1(x)$ that contains $x$\footnote{Recall that a face of a singular leaf $l$ is an embedded $\bR$ in $l$ that bounds one connected component of $P\smallsetminus l$.}, then $Q$ is a union of two rectangles one containing $y$ in its boundary and the other containing $\sigma^p_2(y)$.
\item $(x,Y) = (x, \{y, \sigma^p_2(y), \ldots, \sigma^p_k(y) \})$ and $x=p$ is singular. Then $R$ is a polygonal neighborhood of $x$ and $Q$ is a union of $k$ rectangles centered at $\{y, \sigma^p_2(y), \ldots, \sigma^p_k(y) \})$.
\end{enumerate} 
Moreover, in every cases, we further require that the saturations by $\cF_1$-\emph{faces} of $R$ and $Q$ are equal. (Note that the saturations by $\cF_1$-leaves are often different.)
We refer to Figure \ref{fig_sing_good_neighborhood} for the schematic explanation of how to define the good neighborhoods.

  \begin{figure}[h]
     \centerline{ \mbox{
\includegraphics[width=11cm]{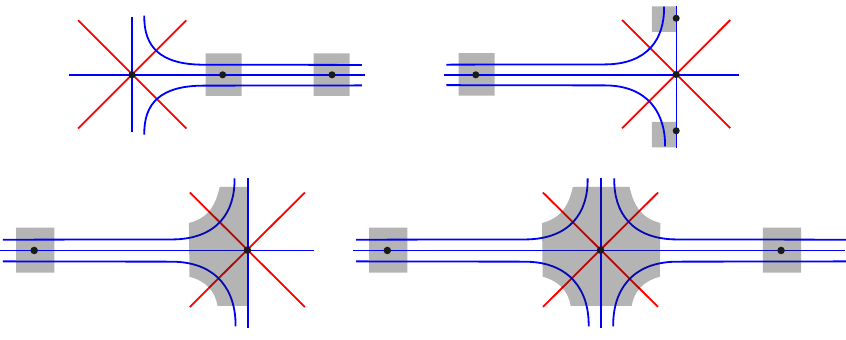}}}
\caption{The good neighborhoods in the singular case}
 \label{fig_sing_good_neighborhood} 
\end{figure}

The proofs of Theorems \ref{thm_compact_singular} and \ref{thm_noncompact} now follows exactly the proof of Theorem \ref{thm_compact} (and the nonsingular case of Theorem \ref{thm_noncompact}) done in Section \ref{sec_nonsingular_main_theorem}. Lemma \ref{lem_torsion_free} remains unchanged, and $W_1^\ast/G$ is now a 3-manifold with a singular foliation with prong singularities; by construction.   

Finally, for Proposition \ref{prop_expansive_1}, the proof follows verbatim, modifying the argument only to replace $y_t$ or $w_t$ with a tuple of points if $x$ or $z$ is on a singular leaf.  This leads to additional (easy) cases to check for the first argument that $z \in \cF_2(x)$, but the definition of good neighborhood has been chosen so the conclusion is nearly immediate.  The only modification required for the argument that $x=z$ is to argue first that $d(\flow^t(p), \flow^{\tau(t)}(q))< \delta$ (when $\delta$ is sufficiently small) implies that they remain in the same trivially foliated flowbox, so no point $u$ in the segment of $\cF_2(x)$ between $x$ and $z$ may have a prong singularity in $\cF_1^>(u)$. See Figure \ref{fig_no_prong_between}.
Thus the construction of an infinite product region proceeds as before, and this completes the proof.  
  \begin{figure}[h]
   \labellist 
  \small\hair 2pt
   \pinlabel $x$ at 45 10
     \pinlabel $y_t$ at 180 10
     \pinlabel $u$ at 45 40
               \pinlabel $z$ at 45 65
     \pinlabel $w_t$ at 180 65
 \endlabellist
     \centerline{ \mbox{
\includegraphics[height=2.5cm]{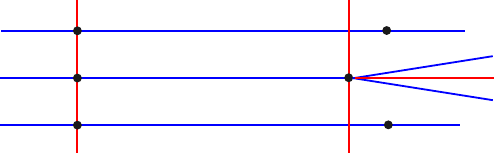}}}
\caption{$\cF_1^>(u)$ cannot contain a prong singularity}
 \label{fig_no_prong_between} 
\end{figure}
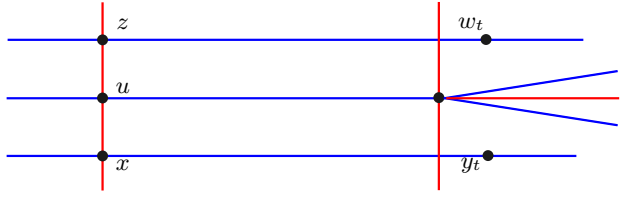

\subsection{The closing property and uniform hyperbolicity on singular planes} \label{sec_closing}
To prove Theorem \ref{thm_closing_implies_prop_disc} in the singular case, we need to first describe the right analogue of the 
closing property in the presence of prongs, and then show that we can recover the setting of the proof of Proposition \ref{prop_proper_discontinuity} even with our more general definition of the space $W_1^\ast$.

\begin{definition}[Closing property, general case]\label{def_closing_prop_general_case}
An action $G$ on a bifoliated plane $(P,\cF_1,\cF_2)$ has the closing property if the following is satisfied:
\begin{itemize}
\item  Each nonsingular point $x \in P$ has a neighborhood basis $U_i$ in $P$, with the property that for each $U_i$ there is a smaller neighborhood $V_i \subset U_i$ such that, if $g(V_i) \cap V_i \neq \emptyset$, then $g$ has a fixed point in $U_i$. 
\item Each singular point $p\in P$ has a pair of neighborhood basis $V_i \subset U_i$ in $P$, such that if $C_1$ (resp.~$C_2$) is any connected component of $V_i\smallsetminus \cF_1(p)$ (resp.~$V_i\smallsetminus \cF_2(p)$), and $g(C_1)\cap C_1 \neq \emptyset$ (resp.~$g(C_2)\cap C_2 \neq \emptyset$), then $g$ has a fixed point inside the connected component of either $U_i\smallsetminus \cF_1(p)$ or $U_i\smallsetminus \cF_2(p)$ intersecting $C_1$ (resp.~intersecting $C_2$).
\end{itemize}
\end{definition}
Notice that the orbit space actions induced from a pseudo-Anosov flow do satisfy this, which follows again from the pseudo-Anosov closing lemma, see \cite[Proposition 1.4.7]{BM_book}.

Recall that the bulk of the proof of Theorem \ref{thm_closing_implies_prop_disc} in the nonsingular case 
was done in Proposition \ref{prop_proper_discontinuity}.  To treat the singular case, we prove the following extension of 
Proposition \ref{prop_proper_discontinuity}.  

\begin{proposition}[Proper discontinuity, general case]\label{prop_proper_discontinuity_general}
Suppose $G$ acts on a bifoliated plane $(P,\cF_1,\cF_2)$ (with or without singularities) such that $G$ has the closing property and 
uniformly hyperbolic fixed points. Then $G$ acts properly discontinuously on the space $W_1^\ast$.
\end{proposition} 

\begin{proof}
The action of $G$ on $W_1^\ast$ is properly discontinuous if and only if for all points $(x,X),(y,Y)$ in $W_1^\ast$, there exists neighborhoods $U_x, U_y$ of $(x,X)$ and $(y,Y)$ respectively such that the number of elements $g\in G$ such that $gU_x\cap U_y \neq \emptyset $ is finite (see e.g., \cite[Theorem 11]{Kapovich}). 

Assuming that this is not the case, we can find $(x,X),(y,Y)$ in $W_1^\ast$, a family of shrinking neighborhoods $U^n_x, U^n_y$ of $(x,X)$ and $(y,Y)$ respectively, and a sequence of distinct elements $g_n\in G$ such that $g_nU^n_x \cap U^n_y \neq\emptyset$.

As there are only countably many singular points (since singular points are closed and discrete, thus in finite number in any compact), the set of singular leaves is countable. Thus we may choose points $x_n$ on nonsingular leaves, such that $(x_n,\{x_n'\})\in U_x^n$ and $g_n(x_n,\{x_n'\})\in U_y^n$.

In particular, we get that $(x_n,x_n')$ converges to points $(x,x')$, with $x'\in \cF_1^>(x)$, and $g_n(x_n,x_n')$ converges to $(y,y')$, with $y'\in \cF_1^>(x)$.

The rest of the proof can now be done as in the proof of Proposition \ref{prop_proper_discontinuity}, using the version of the closing property given in Definition \ref{def_closing_prop_general_case} if one of $x,x',y,$ or $y'$ happen to be a prong singularity. 
\end{proof}

Given this result, the end of the proof of Theorem \ref{thm_closing_implies_prop_disc} follows exactly as in the nonsingular case.  

No modifications are needed for Theorem \ref{thm_transitivity_and_discrete_implies_uniform}, since the proof only involves trivially foliated rectangles, thus the singular analogue of Theorem \ref{thm_closing_plus_transitive_implies_prop_disc} follows from the singular version of Theorem \ref{thm_closing_implies_prop_disc} exactly as before.

Finally, we prove Theorem \ref{thm_converse}.  Note that the nonsingular case is already covered by Theorem \ref{thm_transverse_orientable_model_flow}. Here we give an independent argument which covers both cases, by showing that such actions satisfy the hypotheses of Theorem \ref{thm_closing_implies_prop_disc}.  

\begin{proof}[Proof of Theorem \ref{thm_converse}]
Let $\flow$ be a transversally orientable pseudo-Anosov flow on a compact  
$3$-manifold $M$, we want to show that the induced action of $\pi_1(M)$ on the orbit space $(\orb,\cF^s,\cF^u)$ satisfies the conditions of Theorem \ref{thm_closing_implies_prop_disc}. As mentioned above, the fact that $\pi_1(M)\acts (\orb,\cF^s,\cF^u)$ satisfies the closing property is a consequence of the pseudo-Anosov closing lemma (see, e.g., \cite[Propositions 1.4.4 and 1.4.7]{BM_book}).

What remains is to show that this action also has uniformly hyperbolic fixed points: Consider a rectangle $R$ in $\orb$ and lift it to a two dimensional section $\wt R \subset \wt M$ transverse to the lifted flow $\wt\flow$. If there exists a sequence of distinct elements $g_n\in \pi_1(M)$ with a fixed point in $R$, it corresponds to a sequence $\alpha_n$ of periodic orbits in $M$ with longer and longer periods (or the same orbit being traversed more and more), with lifts $\wt\alpha_n$ all intersecting  $\wt R$. Call $x_n = \wt\alpha_n \cap \wt R$. 

Now assume that $\cF^s(g_nR)$ is a decreasing sequence, such that $g_n$ acts as a contraction on the unstable leaf of $x_n$ (seen in $\orb$). 
This means that any point in $\wt \cF^u(g_n x_n)\cap g_n \wt R $ is contained in an orbit of $\wt \cF^u(x_n)\cap \wt R$, where $\wt \cF^s, \wt \cF^u$ denote the stable and unstable foliations of the flow lifted to the universal cover $\wt M$.  What we need to show is that the intersection of the flow saturation of $\wt \cF^u(g_n x_n)\cap g_n \wt R$ with $\wt \cF^u(x_n)\cap \wt R$ converges to a single point as $n$ goes to $\infty$.

Since $x_n$ is on an orbit invariant by $g_n$, there exists $t_n\in \bR$ such that $g_nx_n = \wt\flow^{t_n}(x_n)$. We claim that $t_n\to +\infty$. First, note that $t_n$ must be positive as $g_n$ was chosen so that it acts as a contraction on the unstable leaf of $x_n$. Then $t_n$ must go to infinity as the $g_n$ are all distinct, and there are only finitely many orbits of bounded period for $\flow$ (recall that we assume here that $\flow$ is on a compact manifold).\footnote{Note that assuming $M$ compact is not actually necessary here, all we need is that the length of periodic orbits hitting the \emph{compact} transversal $R$ goes to $\infty$. This holds as long as the constants of hyperbolicity of the flow are uniform.}
 Hence, the length of any nontrivial segment contained in $\wt \cF^u(x_n)\cap \wt R$ goes to infinity (as $n \to \infty$) when flowed by $\wt\flow^{t_n}$, and this implies that the length of the intersection of the flow saturation of $\wt \cF^u(g_n x_n)\cap g_n \wt R$ with $\wt R$ goes to zero.

Similarly, if $\cF^u(g_nR)$ is a decreasing sequence, then we have $g_nx_n = \wt\flow^{-t_n}(x_n)$ with $t_n\to +\infty$, and the same argument applies.  
 Thus the action $\pi_1(M)\acts (\orb,\cF^s,\cF^u)$ has uniformly hyperbolic fixed points, so Theorem \ref{thm_closing_implies_prop_disc} applies and shows the action is properly discontinuous.  
\end{proof}

\section{Application to loom spaces} \label{sec_applications}
In this section, we describe a natural situation where the assumptions of Theorem \ref{thm_closing_implies_prop_disc} hold, which has recently received significant attention.  
This is the study of {\em loom spaces} associated to veering triangulations. 

A {\em veering triangulation} is a special type of ideal triangulation of a cusped hyperbolic 3-manifold, introduced by Agol \cite{Ago11} and Guéritaud \cite{Gue16}.  Subsequent work, by many authors (\cite{AT24,LMT23,FSS,SS23,SS24}), showed that one can associate a veering triangulation to a transitive pseudo-Anosov flow after drilling along orbits to create cusps, and conversely, one can build a transitive pseudo-Anosov flow on a closed 3-manifold from the data of a veering triangulation on a cusped manifold together with a set of filling slopes.

In the work of Schleimer and Segerman, also with Frankel, \cite{FSS,SS23,SS24} a certain type of bifoliated plane called a \emph{loom space} (see \cite[Definition 2.11]{SS24}) appears as an intermediary structure in their correspondence between veering triangulations and transitive pseudo-Anosov flows, modeled after the orbit space of the punctured pseudo-Anosov flow. Such bifoliated planes also naturally appear in the other proofs of that correspondence, but not always under the name loom space.  These proofs all rely in some way on building a pseudo-Anosov flow from the veering triangulation data via the use of branched surfaces. See \cite[Chapter 2]{Tsa23} for a nice exposition of the different proofs.
Here, we give a different approach, showing that 3-manifolds with expansive flows can be constructed directly from the bifoliated plane (using Theorem \ref{thm_closing_implies_prop_disc}); this is Theorem \ref{thm_loom_space}.  Recall the statement:

\loomthm*

%\begin{remark} \label{rem_replace_condition} 
%The condition ``any nontrivial element of $G$ has at most one fixed point on each leaf" can be replaced with an equivalent condition which is less dynamical and more about the structure of $P$.  To explain this, recall that two leaves $l_1$, $l_2$ or rays of foliations $\cF_1, \cF_2$ plane make a {\em perfect fit} if they do not intersect, but each is contained in the closure of the saturation of the other; i.e. $l_1 \in \partial \cF_1(l_2)$ and $l_2 \in \partial \cF_2(l_1)$.  
%Proposition \ref{prop_points_on_nonseparated} below shows that the ``at most one fixed point on each leaf" is equivalent to
% ``if a leaf makes a perfect fit, then it does not contain a fixed point of a nontrivial element of $G$".  
%\end{remark} 

The conclusion of the veering triangulation to pseudo-Anosov flow correspondence is stronger than the statement of Theorem \ref{thm_loom_space} -- they also show that, given a set of filling slopes, one can fill the cusps of $M$ to obtain a pseudo-Anosov flow on a closed manifold. We do get that the flow near a cusp looks like a punctured neighborhood of a (possibly singular) periodic orbit, see Remark \ref{rem_structure_cusp}.

To begin, we quickly recall the definition of a loom space (see \cite[Definition 2.11]{SS24}), restated in a slightly different terminology.  
\begin{definition}  \label{def_loom_space}
A loom space is a bifoliated plane $(P,\cF_1,\cF_2)$ with no singularities that satisfies the following two properties:
\begin{enumerate}[label=(\roman*)]
\item \label{item_loom_perfect_fit_nonseparated} If two leaves $l_1\in \cF_1$, $l_2\in \cF_2$ make a perfect fit, then there exists $l_1'\in \cF_1$ non-separated with $l_1$ making a perfect fit with $l_2$, and similarly there exists $l_2'\in \cF_2$ nonseparated with $l_2$ making a perfect fit with $l_1$\footnote{This is a rephrasing of the condition ``For every cusp side $s$ of every cusp rectangle, some initial open interval of $s$ is contained in some rectangle'' of \cite[Definition 2.11]{SS24}.}.
\item \label{item_loom_tetrahedron} Each open rectangle (i.e., a trivially bifoliated open set in $P$) is contained in a \emph{tetrahedron rectangle}.
\end{enumerate}
\end{definition}

A {\em tetrahedron rectangle} is an open rectangle such that its closure in the plane is bounded by four ``sides", each side consisting of the union of two rays of nonseparated leaves, as in Figure \ref{fig_veering}. See \cite[Definition 2.9]{SS24}. Following the terminology of \cite{SS24}, we call an ideal point in the closure of a tetrahedron rectangle a \emph{cusp}.

 \begin{figure}
     \centerline{ \mbox{
\includegraphics[width=5cm]{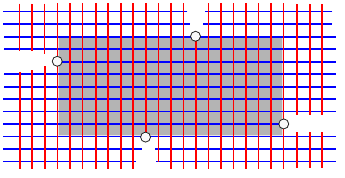}}}
\caption{A tetrahedron rectangle}
 \label{fig_veering} 
\end{figure}

\subsection{Structure of automorphism groups of loom spaces} 
The definition of a loom space constrains both the topology of the foliations and the structure of automorphisms.  
We start by establishing some basic results on this structure. Several of these appear in various forms (occasionally under other assumptions) in the literature. Theorem 15.12 of \cite{BJK25} in particular gives a complete description of the action of an element of $\Aut^+(P)$ for slightly more general bifoliated planes. Since the terminologies are different, and the assumptions not always exactly the same, we provide proofs in order to keep this text self-contained.

\begin{observation}
Let $(P,\cF_1,\cF_2)$ be a loom space.  
Then $(P,\cF_1,\cF_2)$ has no infinite product regions.
\end{observation}
 This is because the interior of an infinite product region is an open rectangle, which by definition cannot be contained in any tetrahedron rectangle.

\begin{observation}\label{obs_no_double_perfect_fits}
Let $(P,\cF_1,\cF_2)$ be a loom space. If a leaf $l_1$ makes a perfect fit with a leaf $l_2$, then there are no leaves making a perfect fit with the other end of $l_1$. In other words, no leaf can have perfect fits on both of its ends.
\end{observation}
By ``perfect fit with an end" we mean, as is standard, a perfect fit with a (or equivalently, any) ray defining that end.  
The proof of this observation follows immediately from conditions \ref{item_loom_perfect_fit_nonseparated} and \ref{item_loom_tetrahedron} of Definition \ref{def_loom_space}: If a leaf made a perfect fit on both of its ends, then condition \ref{item_loom_perfect_fit_nonseparated} would force the existence of a rectangle having two cusps on a single side, contradicting condition \ref{item_loom_tetrahedron}. See \cite[Lemma 2.25]{SS24} for the details.

As a direct consequence of the two previous observations, we also have
 
 \begin{observation} \label{obs_boundary_has_pf} 
Suppose $r, r'$ are rays of leaves of $\cF_1$.  Then $\cF_2(r) \cap \cF_2(r')$, if nonempty, is bounded on one side either by a leaf making a perfect fit with $r$ or $r'$ (but not both), or by a pair of two nonseparated leaves, one intersecting $r$ and one intersecting $r'$, and both making a perfect fit with a common ray between $r$ and $r'$. See Figure \ref{fig_observation_boundary}.
\end{observation} 

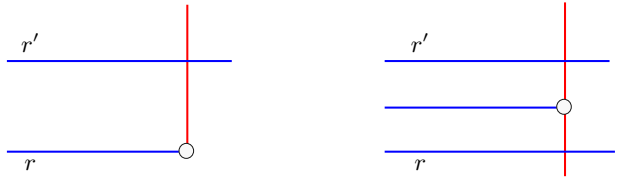
\begin{figure}[h]
   \labellist 
  \small\hair 2pt
     \pinlabel $r$ at 10 5
     \pinlabel $r'$ at 10 57
     \pinlabel $r$ at 175 5
     \pinlabel $r'$ at 175 57
 \endlabellist
     \centerline{ \mbox{
\includegraphics[width=8cm]{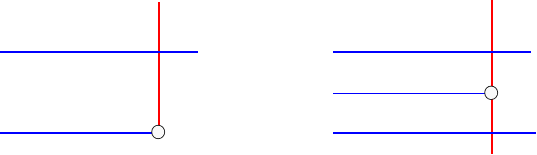}}}
\caption{The two cases, up to symmetry, for the boundary of $\cF_2(r) \cap \cF_2(r')$.}
 \label{fig_observation_boundary} 
\end{figure}

 We will show in Proposition \ref{prop_one_fixed_point} that any non-trivial element of $g$ admits at most one fixed point in the plane. In preparation for this, we first show the following:
\begin{lemma}\label{lem_points_on_nonseparated}
Let $(P,\cF_1,\cF_2)$ be a loom space and $g\in \Aut^+(P)$ nontrivial. 
Suppose $l_1$ and $l_2$ make a perfect fit.  If $g$ fixes a point on $l_1$, then $g$ fixes two distinct points on some leaf. Conversely, if $g$ fixes two points on a single leaf, then there exist leaves $l_1$ and $l_2$ making a perfect fit, such that $g$ fixes a point on $l_1$.
\end{lemma}

\begin{proof}
We start by proving the first direction. To fix notation, assume that $l_1\in \cF_1$, the other case being symmetrical.

Assume that $g$ fixes a point $x\in l_1$. By Observation \ref{obs_no_double_perfect_fits}, at least one of the two rays of $\cF_2(x)\smallsetminus\{x\}$ will not make any perfect fit. Call $r_2$ such a ray. Up to replacing $l_2$ by the other leaf making a perfect fit with $l_1$ (in the notation of the definition of loom space, this is $l_2'$), we can assume that there exist leaves of $\cF_1$ intersecting both $r_2$ and $l_2$.

The set 
$I:=\cF_1(r_2) \cap \cF_1(l_2)$
is bounded on one side by $l_1$, and by Observation \ref{obs_boundary_has_pf}, its other boundary contains leaves which either make 
a perfect fit with $r_2$ (which is impossible) or $l_2$ (which is impossible since $l_2$ only makes a perfect fit with $l_1$), or 
is a union of two nonseparated leaves $l_1',l_1''$ with $l_1'$ intersecting $r_2$ and $l_1''$ intersecting $l_2$.  
Since $r_2$ and $l_2$ are $g$-invariant, $I$ and its boundary are also $g$-invariant. 
Thus the intersection point of $r_2$ and $l_1'$ is fixed by $g$ giving a second fixed point of $g$ on $\cF_2(x)$. See Figure \ref{fig_two_fixed_points}.
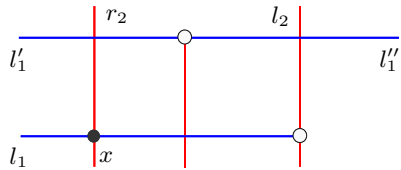
\begin{figure}[h]
 \labellist 
  \small\hair 2pt
     \pinlabel $x$ at 40 5
     \pinlabel $l_1$ at 0 5
     \pinlabel $l_2$ at 120 70
     \pinlabel $r_2$ at 45 70
     \pinlabel $l_1'$ at 0 50
     \pinlabel $l_1''$ at 170 50
 \endlabellist
     \centerline{ \mbox{
\includegraphics[width=5cm]{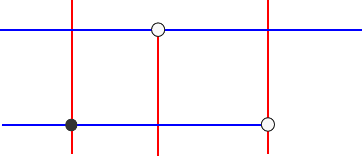}}}
\caption{A second fixed point of $g$.}
 \label{fig_two_fixed_points} 
\end{figure}

Now we show the second direction. Suppose that $g$ fixes two distinct points $x,y$ on a common leaf, without loss of generality in $\cF_1$. If one of $\cF_2(x)$ or $\cF_2(y)$ makes a perfect fit, there is nothing to prove. So we assume that neither makes perfect fits.
Let $r_2$, $r_2'$ denote the positive (with respect to the leafwise orientation) rays of $x$ and $y$ in $\cF_2$, respectively and let $I = \cF_1(r_2) \cap \cF_1(r_2')$.  This is nonempty and bounded on one side by $\cF_1(x) = \cF_1(y)$.

By Observation \ref{obs_boundary_has_pf}, the other boundary contains a leaf $l$ intersecting either $\cF_2(x)$ or $\cF_2(y)$ and making a perfect fit with the other or with some leaf in between.  
Since $r_2$, $r_2'$ and $I$ are $g$-invariant, the boundary leaves are as well.  Thus $\cF_2(x) \cap l$ or $\cF_2(y) \cap l$, whichever 
is nonempty, is a point fixed by $g$ on a leaf $l$ making a perfect fit.  
\end{proof}

\begin{proposition} \label{prop_one_fixed_point}
Let $(P,\cF_1,\cF_2)$ be a loom space, and suppose $G<\Aut(P)$.  Then any nontrivial $g\in G$ has at most one fixed point on any given leaf.   Equivalently, no leaf containing a fixed point of some nontrivial element of $g$ makes a perfect fit with any leaf. 
\end{proposition}

We thank H.~Baik for pointing out that this result as well as some of the consequences follow from the classification of homeomorphisms preserving a {\em veering pair of laminations of $S^1$}, given in \cite[Section 15]{BJK25}.  In order to avoid giving an extensive dictionary between the terminology and assumptions used in \cite{BJK25} and ours, and to keep this work self-contained, we provide standalone proofs here.  

\begin{proof}
Suppose that $g$ fixes two points $x \neq y$ on a leaf $l$.  %Up to passing to a power, we can assume $g$ preserves orientations.  
For concreteness, assume here that $l \in \cF_1$.

Since $g$ fixes both $x$ and $y$ on $l\in \cF_1$, we must have that $g\in \Aut_1^+(P)$. 
As a first case, suppose $g$ also preserves orientation on the $\cF_2$ leaf space.  We will show that it must be the identity. 

To do that, we first find a fixed point on $l$ between $x$ and $y$.  Let $l_x$ and $l_y$ be the leaves of $\cF_2$ through $x$ and $y$ respectively, and consider the boundary of $\cF_1(l_x) \cap \cF_1(l_y)$. On each side, this boundary consists of a leaf segment making a perfect fit with $l_x$ or $l_y$, or a union of two nonseparated leaves, which are fixed by $g$.  In the latter case, there is a leaf $f$ making a perfect fit with each, fixed by $g$, and intersecting $l$ between $x$ and $y$.  This gives the desired fixed point.  
In the second case, up to switching labels assume the boundary leaf $b$ makes a perfect fit with $l_x$.  Then there is a leaf $l_x'$ nonseparated with $l_x$ making a perfect fit with $b$, and we can consider the boundary of $\cF_1(l'_x) \cap \cF_1(l_y)$.  Since $l_2''$ has no perfect fit on its other end, this set is bounded by a $g$-invariant pair of nonseparated leaves making a perfect fit with a leaf between $x$ and $y$ (as desired), or by a leaf making a perfect fit with $l_y$ and a leaf $l_y'$ nonseparated with $l_y$.  In this last case, consider finally $\cF_1(l'_x) \cap \cF_1(l_y')$, which is necessarily then bounded by a $g$-invariant pair of leaves making a perfect fit with a ($g$-invariant) leaf of $\cF_2$ meeting $l$ between $x$ and $y$. See Figure \ref{fig_new_argument}.  

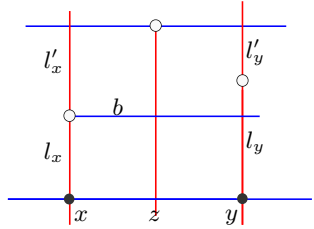
\begin{figure}[h]
 \labellist 
  \small\hair 2pt
    \pinlabel $x$ at 40 5
     \pinlabel $z$ at 80 5
     \pinlabel $y$ at 122 4
          \pinlabel $b$ at 60 65
     \pinlabel $l_x$ at 25 40
     \pinlabel $l_x'$ at 25 90
     \pinlabel $l_y$ at 135 45
     \pinlabel $l_y'$ at 135 95
 \endlabellist
     \centerline{ \mbox{
\includegraphics[width=4cm]{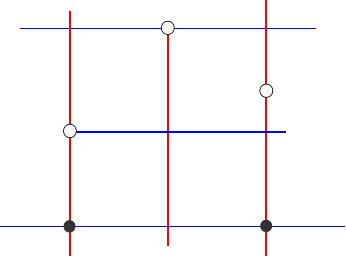}}}
\caption{Obtaining a fixed point $z$ for $g$ between $x$ and $y$}
 \label{fig_new_argument} 
\end{figure}

Since the set of fixed points of any element is closed, the above argument shows that for any leaf $l'$ of either foliation, the set $\mathrm{fix}(g) \cap l'$ is a closed (possibly empty or degenerate) interval, and we have assumed at least one such, the interval $[x,y] \subset l \in \cF_1$, is nondegenerate and nonempty. 

Since leaves that are non-separated with others are countable\footnote{As the leaf space is second countable, it is covered by countably many intervals, and hence there exists only countably many nonseparated leaves.} and that these non-separated leaves are the only ones making perfect fits (by \ref{item_loom_perfect_fit_nonseparated} of Definition \ref{def_loom_space}), we may now assume without loss of generality that $l_x=\cF_2(x)$ does not make any perfect fits.
 Take $x_n \in [x,y]$ approaching $x$ and let $l_n = \cF_2(x_n)$, and $I_n = \cF_1(l_x) \cap \cF_1(l_n)$.  As $x_n$ tends to $x$, the $g$-invariant boundary leaves of $I_n$ will meet $l_x$ arbitrarily far from $x$ (because $l_x$ does not make any perfect fit).  By our argument above, this shows that $l_x$ is pointwise fixed by $g$.  Applying this argument iteratively and using connectedness of $P$, we conclude that $P$ is globally fixed by $g$, as desired. 
 
 Now we finish the proof, treating the case when $g\notin \Aut^+(P)$ (which we will show is impossible). By the above, we know that $g^2\in \Aut^+(P)$ must be the identity. Since $g$ fixes $x,y\in l\in \cF_1$, $g$ must fix every point in $l$.  
 If $g$ is not the identity, then it must be a reflection of $P$ in the line $l$. Considering any tetrahedron rectangle that intersects $l$, we see that this is impossible, as a tetrahedron rectangle has no axial symmetry. 
\end{proof}

From Proposition \ref{prop_one_fixed_point}, we can deduce that fixed points are hyperbolic and unique.
\begin{lemma}\label{lem_unique_fixed_point}
Let $(P,\cF_1,\cF_2)$ be a loom space and  $G<\Aut^+(P)$.
Then any fixed point of a nontrivial element of $G$ is hyperbolic, and each 
nontrivial element fixes at most one point in $P$.
\end{lemma}

\begin{proof}
Suppose $g\in G$ is nontrivial and fixes a point $x\in P$.
By Lemma \ref{lem_points_on_nonseparated}, neither $\cF_1(x)$ nor $\cF_2(x)$ makes a perfect fit with any leaves.
Let $r_1$ and $r_2$ be rays of leaves of $\cF_1(x)$ and $\cF_2(x)$ respectively, based at $x$  and bounding a quadrant $Q$.
We will show that if $g$ is attracting on $r_1$ then it is repelling on $r_2$ and vice versa. This argument applied inductively on each quadrant implies that $x$ is hyperbolic.

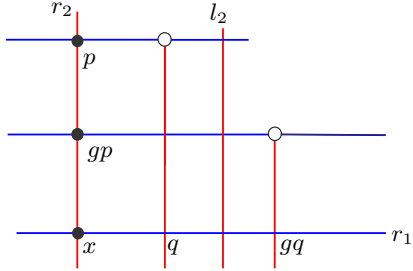
\begin{figure}[h]
 \labellist 
  \small\hair 2pt
     \pinlabel $x$ at 40 10
     \pinlabel $p$ at 40 100
     \pinlabel $q$ at 80 10
     \pinlabel $gp$ at 45 55
     \pinlabel $gq$ at 137 10
     \pinlabel $l_2$ at 102 123
     \pinlabel $r_1$ at 190 15
     \pinlabel $r_2$ at 27 125
 \endlabellist
     \centerline{ \mbox{
\includegraphics[width=5cm]{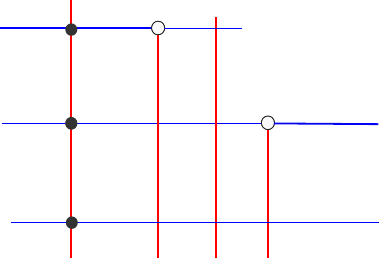}}}
\caption{Fixed points are hyperbolic}
 \label{fig_hyperbolic} 
\end{figure}

Since there are only countably many leaves that are nonseparated with other leaves, and a leaf is nonseparated with another leaf if and only if it admits a perfect fit, there exists some leaf $l_2$ intersecting $r_1$ that does not make any perfect fit. 
Let $I = \cF_1(l_2) \cap \cF_2(r_2)$.  
Since $l_2$ and $\cF_2(x)$ don't make perfect fits with other leaves, Observation \ref{obs_boundary_has_pf} implies 
$I$ is bounded by the union of two nonseparated leaves $l_1,l_1'$ with $l_1$ intersecting $r_2$ at a point $p$ and $l_1'$ intersecting $l_2$. Thus there also exists $l_2'$ a leaf of $\cF_2$ that makes a perfect fit with both $l_1$ and $l_1'$, and $l_2'$ intersects $r_1$ at a point $q$. See Figure \ref{fig_hyperbolic}.
Considering the images of $l_1,l_1'$ under $g$ shows that, either $gp$ is closer to $x$ than $p$, in which case $gq$ must be further from $x$ or vice versa. Since $g$ fixes at most one point per leaf, we conclude that $x$ is a hyperbolic fixed point of $g$.

We now want to show that $g$ fixes at most one point in $P$.  
Suppose for a contradiction that $g$ fixes  distinct points $x$ and $y$. Then $\cF_1(x)$ and $\cF_2(y)$ cannot intersect (otherwise 
$\cF_1(x) \cap \cF_2(y)$ would give a second fixed point for $g$ on $\cF_1(x)$).
Thus, there exists a unique leaf, $l_2$, in the boundary of the $\cF_2$-saturation of $\cF_1(x)$ so that either $l_2$ is equal to $\cF_2(y)$ or $l_2$ separates $\cF_2(y)$ from $\cF_2(x)$. In any case, $g$ must fix $l_2$.  Thus, $\cF_2(x)$ does not make any perfect with any leaves. 
In addition since there are no infinite product regions, either $l_2$ makes a perfect fit with $\cF_1(x)$ or it is nonseparated with a leaf $l_2'$ that makes a perfect fit with $\cF_1(x)$.
Both options are impossible since $\cF_1(x)$ does not make any perfect fits. This contradiction concludes the proof. 
\end{proof}

Similarly to the above result,  an element may fix at most one cusp:

\begin{lemma}\label{lem_fix_unique_cusp}
Let $(P,\cF_1,\cF_2)$ be a loom space and $G<\Aut^+(P)$. If a nontrivial element $g$ preserves two distinct leaves $l,l'$ of the same foliation, then $g$ acts freely on $P$, preserves a cusp $c$ and $l$ and $l'$ both have $c$ as an ideal point.

Consequently, if $g$ fixes a cusp $c$, then $g$ acts freely on $P$, fixes no other cusp, and either preserves all the leaves ending at $c$ or it acts freely on both leaf spaces.
\end{lemma}

\begin{proof}
We assume that $g$ fixes two leaves $l_1,l_1'$ of $\cF_1$; the case of leaves of $\cF_2$ is symmetric.

First, we claim that no leaves of $\cF_2$ can intersect both $l_1$ and $l_1'$: Suppose for a contradiction that the set $\cF_2(l_1) \cap \cF_2(l_1')$ is non-empty. Consider one of the boundaries of $\cF_2(l_1) \cap \cF_2(l_1')$: it is invariant by $g$ and consists (by Observation \ref{obs_boundary_has_pf}) in a pair of nonseparated leaves $l_2,l_2'$ at least one of which must intersect $l_1$ or $l_1'$. Say $l_2$ intersects $l_1$. Then the intersection point $l_1\cap l_2$ is fixed by $g$ and is on $l_2$, a nonseparated leaf, which contradicts uniqueness of fixed points on leaves.

Therefore $\cF_2(l_1) \cap \cF_2(l_1')=\emptyset$. In particular, there exists a unique leaf $l_2$ in $\partial \cF_2(l_1)$ that separates $l_1$ from $l_1'$. So $g$ fixes $l_2$. Moreover, $l_2$ either makes a perfect fit with $l_1$, or it is nonseparated with a leaf intersecting $l_1$. (There are no other possibilities as leaves cannot make a perfect fit on both ends by Observation \ref{obs_no_double_perfect_fits}.)
But in the latter case, we would once again obtain a fixed point of $g$ on a nonseparated leaf, contradicting uniqueness of fixed points. Thus $l_2$ and $l_1$ makes a perfect fit. In particular, $g$ fixes the cusp at the common ends of $l_1$ and $l_2$.

Condition \ref{item_loom_perfect_fit_nonseparated} of the definition of a loom space then tells us that $g$ fixes infinitely many leaves, all ending at that common cusp $c$. If $l_1'$ was not contained in that family of leaves, then the above argument would yield a contradiction, as $\cF_2(l_1')$ could neither intersect the family of leaves ending at $c$, nor be disjoint from it.

To conclude the proof of the lemma, it suffices to remark that if an element $g\in \Aut^+(P)$ fixes a cusp, then it either leaves two leaves invariant, and the above applies, or it must act freely on both leaf spaces.  The statement that $c$ is the unique cusp fixed by $g$ follows from the fact that no leaves have a cusp at each end.  
\end{proof}

\subsection{The closing property and proof of Theorem \ref{thm_loom_space}} 
To prove Theorem \ref{thm_loom_space}, we first establish the closing property, then check all the hypothesis of Theorem \ref{thm_closing_implies_prop_disc}.
 We restrict our attention to the case when $G<\Aut^+(P)$ because any subgroup $G'$ of  $\Aut_1^+(P)$ has an index at most $2$ subgroup in $\Aut^+(P)$, so showing this index-two subgroup acts properly discontinuously on $W_1^+$ implies that $G'$ also does.

\begin{proposition}\label{prop_loom_implies_closing}
Let $(P,\cF_1,\cF_2)$ be a loom space and $G<\Aut^+(P)$.
Then $G$ has the closing property. 
\end{proposition}

In the proof, we will use tetrahedron rectangles to build a closing pair around any point.  To do this we need to understand 
how such rectangles can intersect with their images under elements of $G$.  This is described in the following lemma.

\begin{lemma}\label{lem_types_of_intersections}
Let $(P,\cF_1,\cF_2)$ be a loom space, $R$ a tetrahedron rectangle and suppose $G<\Aut^+(P)$.
Suppose $gR\cap R \neq\emptyset$ for some $g \neq \mathrm{id}$. Then we have one of the following possibilities (See Figure \ref{fig_types_of_intersections}) for the intersection pattern of $gR$ and $R$: 
\begin{enumerate}[label=(\roman*)]
\item \label{item_markovian} 
\textbf{\em Markovian}: up to replacing $g$ with its inverse, we have $\cF_1(R) \subset \cF_1(gR)$ and $\cF_2(R) \supset \cF_2(gR)$ and both containments are strict in the sense that no side of $R$ or $gR$ is contained in a side of the other. In particular, $g$ fixes a unique point in $R$, which must be hyperbolic. 
\item \label{item_weak_markovian} 
\textbf{\em Weakly markovian}: up to replacing $g$ with its inverse, we have $\cF_2(gR)\subset \cF_2(R)$ and $\cF_1(R) \subset \cF_1(gR)$, but $g$ preserves one of the leaves on a side of $R$. In this case $g$ preserves a unique cusp $c$ of $R$ and the side $s$ containing $c$ is either contained in, or contains, $gs$.
\item\label{item_non_markovian} 
\textbf{\em Diagonal}:  $\cF_2(gR)\cap \cF_2(R)$ is a proper subset of both $\cF_2(gR)$ and $\cF_2(R)$, and similarly for the $\cF_1$-saturations. Moreover, there exists two diagonally opposite corners $v_1, v_2$ of $R$ such that $v_1\in gR$ and $g v_2\in R$. 
\end{enumerate}
\end{lemma}
\begin{figure}[h]
     \centerline{ \mbox{
\includegraphics[width=11cm]{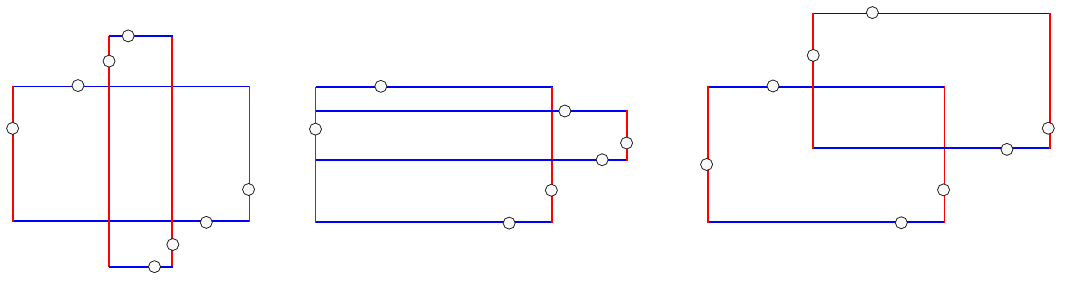}}}
\caption{The possible intersections of $R$ with $gR$: Markovian at left, weakly Markovian at center and diagonal at right.}
 \label{fig_types_of_intersections} 
\end{figure}

\begin{proof} 
Suppose first that $\cF_2(gR) \subset \cF_2(R)$, with strict containment, meaning that the $\cF_2$-sides of $R$ and $gR$ are disjoint.  Since the interior of $R$ cannot contain a side of $gR$ (since any side contains a cusp point), this forces $\cF_1(gR) \supset \cF_1(R)$.
  Again, since interiors of rectangles have no cusps, there are exactly two possibilities for the intersection configuration: Markovian or weakly markovian, depending on whether the containment $\cF_1(gR) \supset \cF_1(R)$ is strict or not, or equivalently, depending on whether $g$ preserves a cusp point or not. 
A symmetric argument applies if $\cF_1(gR) \subset \cF_1(R)$ with strict containment, or if we replace $g$ with $g^{-1}$.

Since $g$ cannot fix two distinct cusps by Lemma \ref{lem_fix_unique_cusp}, $R$ and $gR$ cannot share more than one side. Thus, if none of the cases above occur, then we cannot have any containments of the saturations of $R$, $gR$ or $g^{-1}R$.  This forces a corner of $R$ to be in $gR$ and vice versa, giving case  \ref{item_non_markovian}.
\end{proof}

The case of a diagonal intersection of $R$ and $gR$ does not imply the existence, or lack of, fixed point or cusp for $g$. However, it does imply that some subparts of $R$ are taken off of themselves by $g$ in the following sense:

\begin{observation}\label{obs_side_rectangles_moved_off}
Let $R$ be a tetrahedron and $g\in G$ nontrivial such that $R$ and $gR$ have a diagonal intersection.
Consider the subdivision of $R$ into three (vertical) subrectangles $V_l, V_c, V_r$ obtained by cutting $R$ along the two $\cF_2$-leaves ending at the cusps on the $\cF_1$-sides of $R$. We name those rectangles such that $V_l$ and $V_r$ each have a $\cF_2$-side in common with $R$ and $V_c$ does not.
Then $gV_l \cap V_l =\emptyset$ and $gV_r \cap V_r =\emptyset$.

The same result holds if we instead split $R$ along the two $\cF_1$-leaves ending at the cusps on the $\cF_2$-sides of $R$.
\end{observation}

Notice that in the above result the middle subrectangle may or may not be taken off of itself by $g$, depending on the relative positions of cusps in $R$, see Figure  \ref{fig_intersection_patterns}.
This leads to some challenges in finding neighborhoods for the closing property.  To deal with this, we next show that around any point $x$, we can always build tetrahedron rectangles such that $x$ is {\em not} in the middle subrectangle.

\begin{figure}[h]
     \centerline{ \mbox{
\includegraphics[width=10cm]{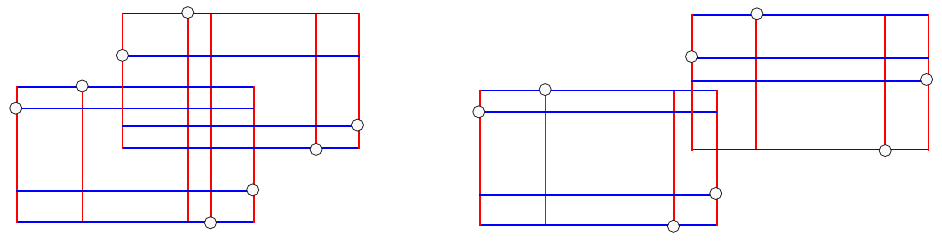}}}
\caption{Diagonal intersection patterns for the two types of tetrahedrons. }
 \label{fig_intersection_patterns} 
\end{figure}

\begin{lemma}\label{lem_building_tetrahedrons}
Let $x\in P$. There exists $V$ a tetrahedron rectangle containing $x$ and such that three of the four cusps are on one side of $\cF_1(x)$. Moreover, we can choose $V$ such that $V\cap \cF_1(x)$ is as small as we like.

Similarly, there exists $H$ a tetrahedron rectangle containing $x$ and such that three of the four cusps are on one side of $\cF_2(x)$, and $H\cap \cF_2(x)$ can be chosen to be inside any fixed neighborhood of $x$.
\end{lemma}

\begin{proof}
We will build the tetrahedron rectangle $H$, the construction for $V$ is symmetric.
Recall that no leaves have perfect fits on both rays. So at least one of the rays of $\cF_1(x)$ based at $x$ does not make a perfect fit, call it $r_1$.

Let $I \subset \cF_2(x)$ be a given neighborhood of $x$ in $\cF_2(x)$.  Let $l_1$ be a leaf of $\cF_1$ (distinct from $\cF_1(x)$) that intersects $I$. Recall from earlier that there are only countably many leaves making perfect fits, so we can choose $l_1$ to not make a perfect fit with any other leaf. 
By Observation \ref{obs_boundary_has_pf}, $\cF_2(l_1) \cap \cF_2(r_1)$ is  bounded by a union of two nonseparated leaves $l_2, l_2'$ with $l_2\cap r_1\neq \emptyset$. These leaves make a perfect fit with some leaf $l_1'$ such that $l_1'$ intersects $I$ between $l_1$ and $x$. See Figure \ref{fig_building_tetrahedron}.

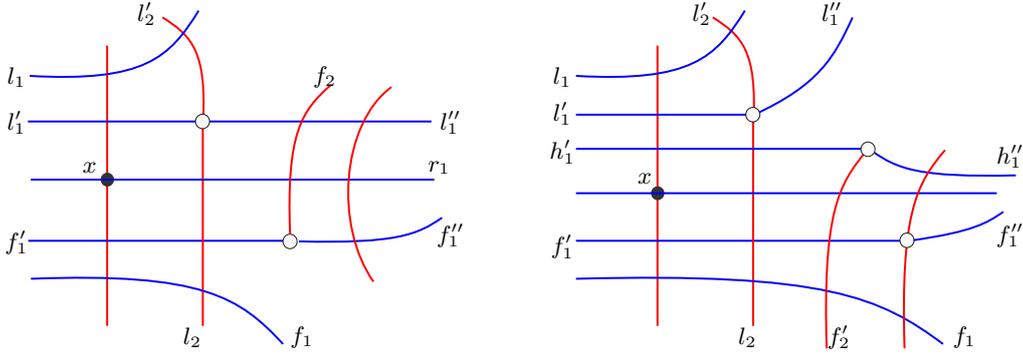
\begin{figure}[h]   \labellist 
  \small\hair 2pt
     \pinlabel $x$ at 27 80
     \pinlabel $l_1$ at -5 120
     \pinlabel $f_1'$ at -5 47
          \pinlabel $f_1''$ at 185 50
          \pinlabel $f_1$ at 120 5
     \pinlabel $l_1'$ at -5 100
     \pinlabel $l_1''$ at 185 100
     \pinlabel $r_1$ at 180 80
    \pinlabel $l_2'$ at 52 148
     \pinlabel $l_2$ at 72 5
     \pinlabel $f_2$ at 130 120
%%%%%     
%    %%%%%
     \pinlabel $x$ at 270 75
     \pinlabel $l_1$ at 234 120
     \pinlabel $f_1'$ at 234 45
          \pinlabel $f_1''$ at 430 50
     \pinlabel $f_1$ at 410 5
     \pinlabel $l_1'$ at 234 104
     \pinlabel $l_1''$ at 352 148
%     \pinlabel $r_2$ at 270 5
    \pinlabel $l_2'$ at 295 148
     \pinlabel $l_2$ at 315 5
    \pinlabel $f_2'$ at 355 5
     \pinlabel $h_1'$ at 234 87
     \pinlabel $h_1''$ at 430 85
 \endlabellist
     \centerline{ \mbox{
\includegraphics[width=13cm]{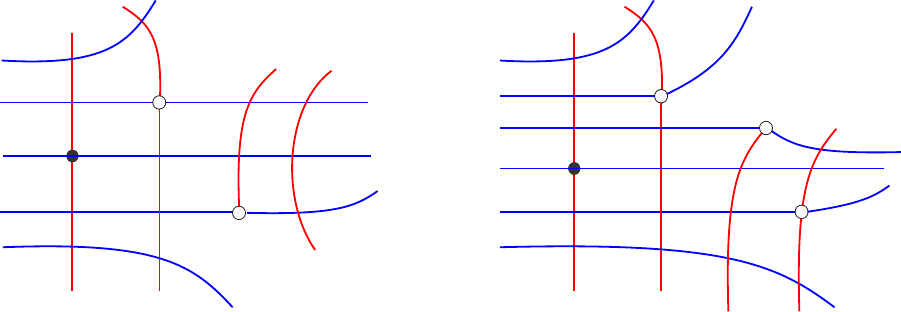}}}
\caption{The two possible configurations when building $H$.}
 \label{fig_building_tetrahedron} 
\end{figure}
Using again condition \ref{item_loom_perfect_fit_nonseparated} of a loom space, the leaf $l_1'$ is nonseparated with a leaf $l_1''$.

Repeat the above construction taking a leaf $f_1$ of $\cF_1$ (distinct form $\cF_1(x)$) that intersects $I$ on the opposite side of $x$ to $l_1$.  We obtain as above two nonseparated leaves $f_2, f_2'$ of $\cF_2$, with $f_2$ intersecting $r_1$, and two nonseparated leaves $f_1'$, $f_1''$, both making a perfect fit with $f_2$, and with $f_1'$ intersecting $I$ between $f_1$ and $x$. See Figure \ref{fig_building_tetrahedron}.

There are now two possible configurations, depending on whether or not there is a common leaf intersecting both $l_1''$ and $f_1''$.
Suppose as a first case that there exists some leaf $l$ intersecting both $f_1''$ and $l_1''$.  Then the union of $f_1$, $f_1''$, $l_1$, $l_1''$, $\cF_2(x)$ and $l$ bound a trivially foliated region.  By assumption, this is contained in a tetrahedron rectangle $H$, and by construction three ideal points are on one side of $x$ and $H\cap \cF_2(x) \subset I$.

So suppose instead that no leaf simultaneously intersects $f_1''$ and $l_1''$.

By construction, the set $\cF_1(l_2)\cap \cF_1(f_2)$ is not empty. It is bounded on one side by $l_2$ and, by Observation \ref{obs_boundary_has_pf}, must be bounded on the other side by a pair of nonseparated leaves, say $h_1', h_1''$, with $h_1'$ intersecting $l_2$, and thus also intersecting $\cF_2(x)$. Necessarily, $h_1' \cap \cF_2(x)$ lies between $x$ and $l_1$, and by construction, leaves close to $f_2$ will simultaneously intersect $f_1''$ and $h_1''$ and we conclude as before. 
\end{proof}

We now have all the ingredients to show that $G$ has the closing property. 
\begin{proof}[Proof of Proposition \ref{prop_loom_implies_closing}]
Let $x\in P$ and let $R$ be a closed, trivially foliated rectangle containing $x$.  
By Lemma \ref{lem_building_tetrahedrons}, we can find a tetrahedron rectangle $H_x$ containing $x$ with $\cF_1(H_x) \subset \cF_1(R)$ and not sharing a side and three cusps of $H_x$ on the same side of $\cF_2(x)$. Similarly, take a tetrahedron $V_x$ containing $x$ with  $\cF_2(V_x) \subset \cF_2(R)$ and not sharing a side and three cusps on one side of $\cF_1(x)$.
Not in particular that $H_x$ and $V_x$ do not have a common cusp.

Let $R_x:= H_x\cap V_x$, and let $g$ be nontrivial.  We will show that, if 
$gR_x\cap R_x \neq \emptyset$ then $g H_x$ and $H_x$ have markovian intersection, as do $g V_x$ and $V_x$.  This implies that $g$ fixes a point in $H_x \cap V_x$, which implies the closing property.\footnote{Notice that this is a strong form of closing property where we can take $(R_x,R_x)$ as a closing pair.}

Since $gR_x\cap R_x$ is non-empty, it implies that both $gH_x\cap H_x$ and $gV_x\cap V_x$ are non-empty. Each intersections can be either markovian, or weakly markovian, or diagonal, given us a priori nine different possibilities.

We start by showing that neither $gH_x\cap H_x$ nor $gV_x\cap V_x$ can be diagonal. This follows immediately from our choice of tetrahedron rectangles and Observation \ref{obs_side_rectangles_moved_off}: Since this is symmetric, we do it for $H_x$. Consider the split of $H_x$ along the $\cF_2$-leaves ending at the cusps on the $\cF_1$-sides into three subrectangles $H_l, H_c, H_r$, where $H_c$ separates $H_l$ from $H_r$. By construction $x$ is not in $H_c$, as $\cF_2(x)$ has three cusps of $H_x$ on one of its side. So $x$, and therefore also $R_x$ is contained in either $H_l$ or $H_r$, say $H_l$. But Observation \ref{obs_side_rectangles_moved_off} gives us that $gH_l\cap H_l$ is empty, a contradiction.

Therefore, there are only three (up to symmetries) possible cases for the type of intersections of $gH_x\cap H_x$ and $gV_x\cap V_x$:
\begin{enumerate}[label=(\alph*)]
\item \label{item_weakly_markovian_twice} both $gH_x\cap H_x$ and $gV_x\cap V_x$ are weakly markovian,
\item \label{item_weakly_markovian_and_markovian}one of $gH_x\cap H_x$ or $gV_x\cap V_x$ is markovian and the other is weakly markovian,
\item \label{item_markovian_twice} both $gH_x\cap H_x$ and $gV_x\cap V_x$ are markovian.
\end{enumerate}

If case \ref{item_weakly_markovian_twice} happens, then we get that $g$ fixes one cusp of $H_x$ and a cusp of $V_x$, but since $H_x$ an $V_x$ do not share a common cusp, this contradicts Lemma \ref{lem_fix_unique_cusp}.

If case \ref{item_weakly_markovian_and_markovian} happens, then $g$ must fix a point in $P$ as well as a cusp, once again contradicting Lemma \ref{lem_fix_unique_cusp}.

Thus we are in case \ref{item_markovian_twice}: both $gH_x\cap H_x$ and $gV_x\cap V_x$ are markovian and hence $g$ fixes a point in $R_x$.
\end{proof}

\begin{rem} \label{rem.cyclic}
While not necessary to prove the theorem, we point out the following fact that is implied by the closing property satisfied by such a group $G$: Suppose that $G<\Aut^+(P)$ satisfies the assumption of Proposition \ref{prop_loom_implies_closing}. Then, \emph{for any leaf $l$ making a perfect fit, $\Stab_G(l)$ is at most cyclic.}

Indeed, by Proposition \ref{prop_one_fixed_point}, $\Stab_G(l)$ acts freely on $l$, so, by H\"older's Theorem, must be semi-conjugated to a group of translations of $\bR$. If $\Stab_G(l)$ is not cyclic or trivial, then it must be a indiscrete subgroup of translations, so as in the proof of Lemma \ref{lem_cyclic_stabilizer}, the closing property implies that some element of $\Stab_G(l)$ has a fixed point in $l$, contradicting that it acts freely.
%
%Notice that in the setting of a group action on a loom space coming from a veering triangulation, 
% the fact that such a $\Stab_G(l)$ is cyclic is taken as one of the assumptions instead -- it is automatic from the fact that the triangulation is an ideal triangulation of the manifold.  
\end{rem}

Now that we obtained the closing property, we are ready to adapt the proof of Theorem \ref{thm_transitivity_and_discrete_implies_uniform} to the loom space situation:
\begin{proposition}\label{prop_loom_space_implies_uniformly_hyperbolic}
Let $P$ be a loom space and $G < Aut^+(P)$. 
Then $G$ has uniformly hyperbolic fixed points.
\end{proposition}

\begin{proof}
First, Lemma \ref{lem_unique_fixed_point} implies that $G$ has hyperbolic fixed points.
Let $R$ be a rectangle and suppose for contradiction that there is an infinite sequence $g_n\in G$ with fixed points in $R$ such that $\cF_1(g_nR)$ is a decreasing sequence.

Assume that $\cF_1(g_nR)$ does not converge to a single leaf, and call $U$ the limit of $g_nR$. Then $\mathring U$ is an open rectangle, so it is contained inside a tetrahedron rectangle $K$. This implies that $\bar U$ (the closure of $U$ in $P$) has at least 2 corners $u_1$ and $u_2$. In
other words at most 2 of the ``corners" (meaning the limits of two sides) of $U$ can be ideal points of the 
boundary of $K$. By construction, $u_1$ and $u_2$ are accumulated upon by sequences $g_nv_1, g_nv_2$ where $v_1,v_2$ are corners of the rectangle $R$. By the closing property, we deduce that for any fixed large enough $m,n$, the element $g_m g_n^{-1}$ has a fixed point as close as we want to $v_1$ and a fixed point as close as we want to $v_2$. In particular, $g_m g_n^{-1}$ has at least $2$ fixed points, contradicting Lemma \ref{lem_unique_fixed_point}.

Therefore, $G$ has uniformly hyperbolic fixed points.
\end{proof}

Given Propositions \ref{prop_loom_implies_closing} and \ref{prop_loom_space_implies_uniformly_hyperbolic}, we can deduce the first part of Theorem \ref{thm_loom_space}. More precisely, if $G<\Aut_1^+(P)$, then call $G_2$ the index at most $2$ subgroup in $G<\Aut^+(P)$. Then $G_2$  satisfies the hypothesis of Theorem \ref{thm_closing_implies_prop_disc}. Therefore, $G_2$, and thus also $G$, acts properly discontinuously on $W_1^>$. Now, if an element $g\in G$ fixes a point in $W_1^>$, then it fixes two points on the same $\cF_1$-leaf, and by Proposition \ref{prop_one_fixed_point}, it must be the identity. Thus $G$ acts freely on $W_1^>$. This shows that $W_1^>/G$ is a manifold.  What remains is to show that it is atoroidal.   To do this, we first show that $\bZ^2$
or $\bZ \rtimes \bZ$  subgroups of $\Aut^+_1(P)$ preserve cusps of $P$, then describe the structure of $W_1^>/G$ corresponding 
to such a cusp.  Since it is possible that $G$ does not preserve orientation, we treat the case of Klein bottles as well.

\begin{proposition} \label{prop_atoroidal}
Let $H < \Aut^+_1(P)$ be isomorphic to $\bZ^2$ or $\bZ \rtimes \bZ$.  Then $H$ fixes a unique cusp in $(P,\cF_1,\cF_2)$.
\end{proposition}

The proof uses basic arguments in the theory of actions on bifoliated planes. 
The argument is very similar to the proof of \cite[Proposition 3.2.1]{BM_book}, we give the outline and indicate all points where the proof differs.   See also \cite[Theorem 15.18]{BJK25}.

\begin{proof}
Suppose $H<\Aut^+_1(P)$ is isomorphic to $\bZ^2$.  We will show that $H$ acts freely on $P$, and fixes a unique cusp.  The case of $\bZ \rtimes \bZ$ is then an immediate consequence, since the index 2, $\bZ^2$ subgroup, will fix a unique cusp, which must then necessarily be preserved by the full group. 

Let $g \in H \cong \bZ^2$ be nontrivial.  If $g$ fixes a point, it fixes only one point by Lemma \ref{lem_unique_fixed_point}.
Since $H$ is abelian, this point is invariant and hence fixed by all of $H$. However, since $G$ has the closing property (by Proposition \ref{prop_loom_implies_closing}) and hyperbolic fixed points (by lemma \ref{lem_unique_fixed_point}), Lemma \ref{lem_cyclic_stabilizer} applies and says that point stabilizers are cyclic, a contradiction.  

Thus $H$ acts freely on $P$.  We now want to show it fixes a leaf.  Suppose for contradiction that $H$ fixes no leaf, so its action on each leaf space is free.  
By \cite{Bar98b}, the action of each element $h\in H$ on the leaf space thus has an {\em axis}, which is either an invariant embedded copy of $\bR$ or a $\bZ$-union of (possibly degenerate) intervals $[a_i, b_i]$ where $a_i$ is nonseparated with $b_{i-1}$, along which $h$ acts by translation. In the latter case, $h$ would either shift the intervals $[a_i, b_i]$, or possibly be an orientation reversing shift.  Since $H$ is abelian, the axes of all elements coincide.  Thus, in the latter case where an axis is of the form $\bigcup_{i\in \bZ} [a_i, b_i]$, the induced action on $\bZ$ has nontrivial kernel and thus some element preserves each interval, hence fixes each leaf $a_i$.  This contradicts our assumption that $H$ acts freely on leaf spaces.  

Thus, we are left to deal with the case where the common axis for all elements of $H$ in the leaf spaces of $\cF_1$ and $\cF_2$ respectively are both embedded copies of $\bR$.  Denote these axes by $A_1$ and $A_2$, respectively, and consider the set $\Omega := \{ l_1 \cap l_2 : l_i \in A_i \} \subset P$, which is invariant under $H$.  Since $A_1$ and $A_2$ are both homeomorphic to $\bR$, the set $\Omega$ is isomorphic, as a bifoliated space, to a subset of $\bR \times \bR \cong A_1 \times A_2$ with standard coordinate foliations.  

One can now argue as in \cite[Proposition 2.11.9]{BM_book} that $\Omega$ is non-empty, connected, and each leaf of each $A_i$ meets it along an interval. (This uses only topology of the bifoliated plane and the fact that $H$ acts freely on its axes).
Moreover, the axis $A_i$ are properly embedded, as otherwise their boundary would contain infinitely many pairwise nonseparated orbits, which is impossible in a loom space.
  Thus, considered as a subset of $\bR \times \bR \cong A_1 \times A_2$, $\Omega$ is a connected region bounded above and below by the graphs of monotone (possibly discontinuous and weakly monotone) functions $s$ and $i$ (respectively) from $\bR$ to $\bR \cup \{\pm \infty\}$\footnote{Since the maps may be discontinuous, one can add vertical boundaries to the graphs at each discontinuity jumps in order to get the full boundary of $\Omega$.}. We now fix an orientation so that the upper boundary function $s$ is (weakly) increasing.  If $s$ takes the value $\infty$ somewhere, then (since it is weakly increasing) $s=\infty$ on some open interval $(a, +\infty)$. In that case, either the leaves of $A_2 \cong \bR$ exit every compact set as their parameter increases to $\infty$, giving an infinite product region (which is forbidden), or they limit to a leaf or pair of leaves which are necessarily invariant by $H$, contradicting the assumption of a free action on the leaf spaces.   Thus, we conclude the boundary function takes only finite values.   

Now, if for some $a \in A_1$ we have that $s(a) < s(b)$ when $a<b$, then the leaf $s(a)$ is the limit of leaves in $A_2$ intersecting $a$ and  must make a perfect fit with $a$.  Since loom spaces have only countably many leaves making perfect fits (because leaves making perfect fits are nonseparated with other leaves), we conclude that $s$ must have intervals on which it is locally constant.   Using the fact that the (free) action of $H$ is semi-conjugate to an action by a dense subgroup of translations and preserves the set on which $s$ is locally constant, one can use the proof of \cite[Proposition 3.2.1]{BM_book} to find a point $x \in P$ and  $h_k \in H$, with $\cF_1(x) \in A_1$ and such that $h_k(x)$ converges to $x$.  By the closing property, this gives a fixed leaf in $A_1$ for all $h_k$ with $k$ sufficiently large, which again is a contradiction with our assumption. 

We conclude that $H$ does not act freely on both leaf spaces, hence, there exists $h\in H$ that fixes some leaf $l$ either in $\cF_1$ or $\cF_2$.  

Since $H \simeq \bZ^2$, the set of leaves fixed by $h$ is invariant under $H$. If this set of
leaves is finite, then a finite index subgroup $H' < H$ would globally fix a leaf $l$. Notice that $H'$ must act freely on $l$, since $H$ acts freely on the whole plane $P$.  
Then using H\"older Theorem and the closing property, as in Remark \ref{rem.cyclic}, produces a contradiction.
We conclude that $h$ fixes infinitely many leaves.  In particular, Lemma \ref{lem_fix_unique_cusp} implies that $h^2\in \Aut^+(P)$ fixes a unique cusp, and thus so does $h$.  Since $H$ is abelian, this unique cusp is preserved by all elements of $H$, which is what we needed to show. 
\end{proof}

As a converse to Proposition \ref{prop_atoroidal}, we have the following.
\begin{proposition} \label{prop_cusp_stabilizers}
Let $(P,\cF_1,\cF_2)$ be a loom space.  Let $H < \Aut(P)$ be the stabilizer of a cusp, and assume $H$ nontrivial. Then, up to passing to an index 2 subgroup, $H$ is isomorphic to $\bZ$ or  $\bZ \times \bZ$. 
Furthermore, if $g \in G$ satisfies
$g^{-1} h g = h^{\pm 1}$ for some nontrivial $h \in H$, then $g \in H$.
\end{proposition}

\begin{proof}
Let $c$ be a cusp. 
The leaves of $\cF_1$ and $\cF_2$ ending at $c$ have a natural linear order coming from their embedding in $P$ and each element of $H$ either preserves or reverses this, depending on whether it preserves or reverses orientation on $P$.  Hence, one gets a homomorphism $H \to \bZ \rtimes \bZ/2\bZ$.  The kernel of this homomorphism fixes all leaves ending at $c$; we will next show that it is trivial or $\bZ$.  This follows by the same argument as in Lemma \ref{lem_cyclic_stabilizer}: the stabilizer of a leaf ending at $c$ is abelian by H\"older's theorem, since no element may fix a point and a cusp, and if non-cyclic an accumulation point gives the existence of a fixed point, contradicting the closing property. 

Note also that $\Aut(P)$ contains no order 2 element reversing orientation since tetrahedron rectangles have no reflective symmetry, thus the case $H = \bZ/2\bZ$ is excluded.  Thus, we conclude that the orientation-preserving subgroup of $H$ is isomorphic to $\bZ$ or $\bZ^2$. 

To prove the last statement, note that if $h g = g h^{\pm 1}$, and $h^{\pm 1}$ 
preserves $c$, then $h$ preserves the cusp $g(c)$.
By Lemma \ref{lem_fix_unique_cusp}, $h$ preserves a 
single cusp, and thus $g(c) = c$.  
\end{proof} 

Note that the case of a $\bZ$ stabilizer may indeed occur, for instance if one takes a pseudo-Anosov flow without perfect fits on a compact 3-manifold, drills out the singular orbits, and lifts to a cover that unwraps one of the torus boundary components to a cylinder.  
The orbit space is then a loom space and the stabilizer of the cusp
associated with the cylinder is $\bZ$.

\begin{rem}[Structure of cusp neighborhoods]  \label{rem_structure_cusp}
One can be slightly more precise about the description of the dynamics of the orbits of the flow near a cusp in the following sense:
If $H<\Aut_1(P)$ is a $\bZ^2$ or $\bZ\rtimes\bZ$-subgroup fixing a cusp $c$ in $P$, then we can consider the open $H$-invariant set $U$ defined by taking the union of all the $\cF_1$-leaves ending at $c$ together with the $\cF_1$-saturation of the $\cF_2$-leaves ending at $c$. Then, one can see that $H$ acts properly discontinuously on $N=(U\times U)\cap W_1^>$ and that the dynamics of the flow obtained on $N/H$ corresponds to that of a neighborhood (foliated by complete orbits) of a regular or prong periodic orbit where the periodic orbit is removed. More precisely, one can build an orbit equivalence between the flow on $N/H$ and a punctured neighborhood of a periodic orbit in the same way as in Theorem \ref{thm_transverse_orientable_model_flow}.
\end{rem}

Finally, we conclude the proof of Theorem \ref{thm_loom_space}, by describing the cusp ends of $M = W^>_1/G$ in the case where $G$ is finitely generated.  

\begin{proposition}[Cusps descend to cusps] 
Let $(P,\cF_1,\cF_2)$ be a loom space, $G < \Aut^+_1(P)$ be a finitely generated group, and $H < G$ be a maximal (with respect to inclusion) subgroup isomorphic to $\bZ^2$ or $\bZ \rtimes \bZ$.  Then, $W^>_1/G$ has an end homeomorphic to $T^2 \times [0, \infty)$ or $K \times [0, \infty)$ (respectively, where $K$ is the Klein bottle) preserved by $H$, and such that $\pi_1(T \times \{0\})$ (resp.~$\pi_1(K\times\{0\})$) is $H$.  
\end{proposition}
 
\begin{proof} 
By Proposition \ref{prop_atoroidal}, $H$ preserves a unique cusp.  Since $H$ is assumed maximal, 
Proposition \ref{prop_cusp_stabilizers} implies that $H$ is the cusp stabilizer.  
By Remark \ref{rem_structure_cusp}, the space $W_1^>/H$ (which is a cover of $W_1^>/G$) has an end homeomorphic to a neighborhood of a regular or prong periodic orbit where the periodic orbit is removed.  Let $h\colon T \to W_1^>/H$ be an embedding of a torus or Klein bottle carrying the topology, and let $S'$ denote its image in $M =W_1^>/G$, which is a $\pi_1$-injectively immersed torus or Klein bottle.  

Since $G$ is finitely generated and $M$ is aspherical, the JSJ decomposition applies (See \cite[Theorem 3.4, 3.5]{Bon02} for a statement applicable in this level of generality, and \cite{BS87} for a proof), and 
 $S'$ is homotopic either into a Seifert fibered space or into a neighborhood of a boundary
component of an atoroidal piece. 

Suppose first that $S'$ is  homotopic into a Seifert piece $M_0$
of the torus decomposition.
Since $H$ is a subgroup isomorphic to either $\bZ^2$ or
$\pi_1(K)$, it follows that $H$ contains a power $\delta^n$ of the
element $\delta$ of $\pi_1(M_0)$ which represents
a regular fiber of the piece. 
The last statement of Proposition \ref{prop_cusp_stabilizers} implies that $\delta$
is in $H$. Then since for every $g$ in $\pi_1(M_0)$
we have $g^{-1} \delta g = \delta^{\pm 1}$ it follows
that $\pi_1(P) < H$, which is absurd.  

Thus we have that $S'$ is homotopic into a neighborhood of a boundary
component $S$ of an atoroidal piece, which is an embedded surface. The boundary $S$ consists of elements which commute or anticommute with elements in $\pi_1(S')$. By Proposition \ref{prop_cusp_stabilizers} and maximality of $H$ we have $\pi_1(S) = H$. Thus, $S'$ is actually 
homotopic to $S$. 

It remains to show that $S$ is peripheral.  
Let $\wt{S}$ be the lift of $S$ to $W_1^>$, and $\hat{S}$ the intermediate lift to $W_1^>/H$, which is homotopic to $h(T)$.  Thus, $\hat{S}$ separates $W_1^>/H$ into two connected components.  

By the description in Remark \ref{rem_structure_cusp}, one of the connected component of $W_1^>\smallsetminus \wt S$, call it $A$, is such that its projection $\bar A$ to $W_1^>/H$ is on the side of $\bar S$ that bounds the cusp $c$. From the description of the dynamics of the flow given in that remark, we also get that if a forward orbit $\{(x,y_t), t>0\}$ is entirely contained in $A$, then $x$ must be on one of the $\cF_1$-leaves that ends at $c$.

Now assume for a contradiction that $S$ is not peripheral. Then it is the common boundary torus of two distinct, or the same, atoroidal pieces. Hence, we can find an element $g\in G\smallsetminus H$ such that $g (A)\subset A$: If $S$ is nonseparating, just consider a loop in $M$ intersecting $S$ exactly once, and take $g$ the associated deck transformation; if $S$ is separating, consider instead a loop obtained by concatenating a (non-trivial) loop on one side of $S$ with a (non-trivial) one on the other side, and again take $g$ to be the associated deck transformation.

Now considering a forward orbit $\{(x,y_t), t>0\}$ entirely contained in $A$, we get that $\{g(x,y_t), t>0\}$ is contained in $g(A)$ so is in particular contained in $A$, thus $gx$ must be on one of the $\cF_1$-leaves ending at the cusp $c$. This implies that $gc=c$, so $g\in H$, a contradiction.
We conclude $S$ is peripheral, as desired.
\end{proof}

\bibliographystyle{amsalpha}
\bibliography{orbit_space_ref}

\end{document}